\documentclass[a4paper,12pt]{article}
\usepackage[top=2.5cm,bottom=2.5cm,left=2.5cm,right=2.5cm]{geometry}
\usepackage{amssymb}
\usepackage{graphicx}
\usepackage{amsmath,amsthm,amssymb,lineno}
\usepackage{latexsym}
\usepackage{graphicx,booktabs,multirow}
\usepackage{tikz}
\usetikzlibrary{decorations.pathreplacing}
\usetikzlibrary{intersections}
\usepackage{graphicx,booktabs,multirow}
\usepackage{appendix}

\newtheorem{theorem}{Theorem}
\newtheorem{lemma}[theorem]{Lemma}
\newtheorem{corollary}[theorem]{\rm\bfseries Corollary}

\begin{document}
	\title{Higher order invariants of a graph based on the path sequence}
	\author{Yirong Cai, Zikai Tang, Hanyuan Deng\thanks{Corresponding author:
			hydeng@hunnu.edu.cn}\\
		{\footnotesize College of Mathematics and Statistics, Hunan Normal University, Changsha, Hunan 410081, P. R. China} \\
	}
	\date{}
	\maketitle
	
	\begin{abstract}
		Let $G=(V,E)$ be a simple and connected graph. A $h$-order invariant of $G$ based on the path sequence is defined from a set of real numbers ${f(x_{0},x_{1},\cdots,x_{h})}$ as $^{h}I_f(G)=\sum\limits_{v_{0}v_{1}v_{2}\cdots v_{h}}f\left(d_{0},d_{1},\cdots,d_{h}\right)$, where the sum runs over all paths $v_{0}v_{1}v_{2}\cdots v_{h}$ of length $h$ and $d_{i}$ is the degree of vertex $v_i$ in $G$. In this paper, we first show that the $h$-order invariant of a starlike tree $S_{n}$ can be determined completely by its branches whose length does not exceed $h$. And then we find conditions on the function $f$ for some graph families $\mathcal{G}$ such that any graph $G\in\mathcal{G}$ can be determined by the higher order invariants $^{h}I_f(G)$ for $0\leqslant h\leqslant \rho$, where $\rho$ is the length of a longest path in $G$.
		\end{abstract}
		\maketitle
	
	\makeatletter
	\renewcommand\@makefnmark%
	{\mbox{\textsuperscript{\normalfont\@thefnmark)}}}
	\makeatother
	
	\section{Introduction}\label{intro}
	
	Graph invariant or parameter is one of the fundamental concepts in graph theory, which refers to the number of features of a graph that remains unchanged in the sense of combinatorial isomorphism. The number of vertices and the number of edges of a graph are two simple graph invariants, and the degree sequence of a graph is also a graph invariant. A graph parameter or graph invariant can reflect the structural properties of a graph and is of great significance in graph theory. In particular, the graph invariant of a molecular structure in chemistry is also called a molecular descriptor or topological index. A topological index is a numerical method of molecular structure, which is an invariant of molecular graphs and directly generated from the molecular structure, it reflects the structural characteristics of the compound. Because the molecular topological index is easy to calculate, objective and not limited by experience and experiment, it is widely used. The construction of QSPR/QSAR model based on topological indices \cite{balaban2000historical,devillers2000topological}, physicochemical properties and activity invariants of compounds, and the evaluation and prediction of their properties have become one of the most active fields in mathematics chemistry, and have been paid more and more attention by mathematicians and chemists \cite{betancur2015vertex,gutman2013degree,rada2013vertex}.
	
	Among all topological indices, the topological indices based on vertex degree have attracted much attention in recent studies \cite{dovslic2011vertex,furtula2013structure,gutman2013testing,horoldagva2011some,rada2013vertex}, which are defined by vertex degree of molecular graphs. Given a graph $G$ with $n$ vertices, a vertex-degree-based topological index \cite{betancur2015vertex} is defined from a set of real numbers ${\varphi_{i,j}}(1\leqslant i\leq j \leqslant n-1)$ as
	$$I(G)= \sum_{1\leqslant i\leqslant j \leqslant n-1}m_{ij}\varphi_{i,j}$$
	where $m_{ij}$ denotes the number of edges between vertices of degree $i$ and degree $j$. In other words,
	$$I(G)= \sum_{uv\in{E}}\varphi(d(u),d(v)),$$
	it is a invariant of $G$ based on the edge-degree sequence. Different choices of ${\varphi(x,y)}$ give different topological indices. For example, it is the Randi\'{c} index for $\varphi(d(x),d(y))=\frac{1}{\sqrt{d(x)d(y)}}$; it is the sum-connectivity index for $\varphi(d(x),d(y))=\frac{1}{\sqrt{d(x)+d(y)}}$; it is the harmonic index for $\varphi(d(x),d(y))=\frac{2}{d(x)+d(y)}$.

A path of length $h$ in $G$ is a sequence $v_0v_1\cdots v_h$, where the vertices $v_0,v_1,\cdots,v_h$ are distinct, $v_{i-1}$ and $v_{i}$ are adjacent, $i=1,2,\cdots,h$ (see \cite{Bondy1976}). In this paper, $v_0v_1\cdots v_h$ and its inverse $v_hv_{h-1}\cdots v_0$ are counted as one path.
	
	For an integer $h\geq 0$, the $h$-connectivity index of a graph $G$ is defined as
	$$^{h}\chi(G)=\sum\limits_{v_{0}v_{1}v_{2}\cdots v_{h}}\frac{1}{\sqrt{d_{0}d_{1}\cdots d_{h}}},$$
	where the sum runs over all paths $v_{0}v_{1}v_{2}\cdots v_{h}$ of length $h$ and $d_{i}=d_{G}(v_i)$ is the degree of vertex $v_i$ in $G$. The higher order connectivity indices are of great interest in the theory of molecular graph theory and some of its mathematical properties have been reported, see \cite{rada2002higher} and the cited literature. 	
	
	In this paper, we generalize the vertex-degree-based topological index and the higher order connectivity indices, and introduce a higher order invariant based on the path sequence of a graph. Let $G$ be a simple and connected graph with vertex set $V(G)$ and edge set $E(G)$. A $h$-order invariant of $G$ based on the path sequence is defined from a set of real numbers ${f(x_{0},x_{1},\cdots,x_{h})}$ with ${f(x_{0},x_{1},\cdots,x_{h})}={f(x_{h},x_{h-1},\cdots,x_{0})}$ as
	$$^{h}I_f(G)=\sum\limits_{v_{0}v_{1}v_{2}\cdots v_{h}}f\left(d_{0},d_{1},\cdots,d_{h}\right),$$
	where the sum runs over all paths $v_{0}v_{1}v_{2}\cdots v_{h}$ of length $h$ and $d_{i}=d_{G}(v_i)$ is the degree of vertex $v_i$ in $G$.
	
	For a path $v_{0}v_{1}\cdots v_{h}$ with length $h$ in $G$, the sequence $\left(d_{0},d_{1},\cdots,d_{h}\right)$ is called the degree sequence of $v_{0}v_{1}\cdots v_{h}$. Then the $h$-order invariant $^{h}I_f(G)$ based on the path sequence can also be written as
	$$^{h}I_f(G)=\sum_{\left(d_{0},d_{1},\cdots,d_{h}\right)}\omega\left(d_{0},d_{1},\cdots,d_{h}\right)f\left(d_{0},d_{1},\cdots,d_{h}\right),$$
	where $\omega\left(d_{0},d_{1},\cdots,d_{h}\right)$ denotes the number of paths with the degree sequence $\left(d_{0},d_{1},\cdots,d_{h}\right)$ in $G$.
	
	For $h=0$, the zero-order invariant of $G$ is $^{0}I_f(G)=\sum\limits_{v\in V(G)}f\left(d(v)\right)$. It can be seen as a invariant or topological index based on the degree sequences of a graph $G$. If we take $f(x)=x$, then $^{0}I_f(G)=\sum\limits_{v \in V(G)}d(v)=2|E(G)|$; If we take $f(x)=x^{\alpha}$, where $\alpha$ is a real number, then $^{0}I_f(G)=\sum\limits_{v\in V(G)}d^{\alpha}(v)$ is the zeroth-order general Randi\'{c} index \cite{li2005}.
	
	For $h=1$, the $1$-order invariant of $G$ is $^{1}I_f(G)=\sum\limits_{uv \in E(G)}f\left(d(u),d(v))\right)$. It is just the vertex-degree-based topological index \cite{betancur2015vertex}.
	
	If we take the function $f\left(x_{0},x_{1},\cdots,x_{h}\right)=\frac{1}{\sqrt{x_{0}x_{1}\cdots x_{h}}}$, then $^{h}I_f(G)$ is the $h$-order connectivity index \cite{rada2002higher}. In 2002, Rada and Araujo \cite{rada2002higher} showed that the higher order connectivity index of a starlike tree is completely determined by its branches of length $\leqslant h$, and starlike trees which have equal $h$-connectivity index for all $h\geqslant 0$ are isomorphic. If $f\left(x_{0},x_{1},\cdots,x_{h}\right)\equiv1$, then $^{h}I_f(G)$ is the number of the paths with length $h$ \cite{cai2024}. Recently, Cai and Deng \cite{cai2024} showed that some graphs can be determined by the path sequence. Combining these results, we will first consider the $h$-order invariant of a starlike tree $S_{n}$ and show that it is determined entirely by its branches whose length does not exceed $h$. Then, we explore the conditions on the function $f: \bigcup_{i=1}^{\infty} \mathbb{Z}_{+}^{i} \rightarrow \mathbb{R}$ for some graph families $\mathcal{G}$ such that any graph $G\in\mathcal{G}$ can be determined by the higher order invariants $^{h}I_f(G)$ for $0\leqslant h\leqslant \rho$, where $\rho$ is the length of a longest path in $G$.

	\section{Properties of the higher-order invariants for some graphs}
	
	In this section, we first consider the higher-order invariants $^{h}I_f$ of a starlike tree related to its branches. Let $S_{n}$ be a starlike tree with $n$ vertices and root (i.e., center) $v_{0}$, the degree $d(v_0)=m>2$. We denote by $L_{l}$ the number of $l$-branches in $S_{n}$, where a $l$-branch is a path with length $l$ from the root $v_0$ to a pendant vertex, $l=1,\cdots,t$ and $t$ represents the length of a largest branch in $S_{n}$. Clearly, a starlike tree $S_n$ is completely determined by its $L_1,L_2,\cdots,L_t$. In fact, Rada and Araujo gave a method for calculating the number of paths in a starlike tree in the proof of Theorem 2.3 in \cite{rada2002higher}. This method is classified and calculated based on the sequence of degrees of vertices on the path.
	
	For a graph $G$, $\mathcal{P}_{h}(G)$ denotes the set of all paths with length $h$ in $G$, $\mathbb{N}=\{0,1,\cdots\}$ the set of natural numbers. The function $\Psi_{h} \,\,(\,or\,\, \Psi_{G} ): \mathcal{P}_{h}(G) \rightarrow \mathbb{N}^{h+1}$ is defined by $\Psi_{h}(\pi)=\left(d(v_{1}),d(v_{2}),\cdots,d(v_{h+1})\right)$, where $\pi=v_{1}v_{2}\cdots v_{h+1}$ is a path with length $h$ in $G$. So, a path in $\mathcal{P}_{h}(G)$ can be acted as a natural number sequence (i.e., degree sequence of the path) by the function $\Psi_{h}$. Thus, the number of paths is equivalent to the number of original images of degree sequences under $\Psi_{h}$.
	
	From the proof of Theorem 2.3 in \cite{rada2002higher}, we can get
	
	\begin{lemma} \cite{rada2002higher}\label{l-1}
		The path of length $h$ in $S_{n}$ can be counted by the following three types.\\
		\textbf{Type 1}
		The possible image under $\Psi_{h}$ of all paths in $\mathcal{P}_{h}(S_{n})$ which contain $v_{0}$ as an end-vertex is $X_{1}=(m, \underbrace{2,\cdots,2}_{h-1}, 1)$ or
		$X_{2}=(m, \underbrace{2,\cdots,2}_{h})$, then the number of paths in $\mathcal{P}_{h}(S_{n})$ whose images under $\Psi_{h}$ is $X_{1}$ or $X_{2}$ are
		\[ \left|\Psi_{h}^{-1}\left(X_{1}\right)\right|=L_{h}, \]
		\[\left|\Psi_{h}^{-1}\left(X_{2}\right)\right|=L_{h+1}+\cdots+L_{t}=m-\sum_{i=1}^{h} L_{i}.\]
		\textbf{Type 2}
		The possible image under $\Psi_{h}$ of all paths in $ \mathcal{P}_{h}(S_{n})$ which do not contain $v_{0}$ is $Y_{1}=(\underbrace{2,\cdots,2}_{h+1})$ or $Y_{2}=(1,\underbrace{2,\cdots,2}_{h})$,  then the number of paths in $\mathcal{P}_{h}(S_{n})$ whose images under $\Psi_{h}$ is $Y_{1}$ or $Y_{2}$ are
		\[\left|\Psi_{h}^{-1}\left(Y_{1}\right)\right|=(n-1)-\sum_{i=1}^{h}iL_{i}-(h+1)\left(m-\sum_{i=1}^{h}L_{i}\right),\]
		\[\left|\Psi_{h}^{-1}\left(Y_{2}\right)\right|=L_{h+1}+\cdots+L_{t}=m-\sum_{i=1}^{h}L_{i}. \]
		\textbf{Type 3}
		The possible image under $\Psi_{h}$ of all paths in $ \mathcal{P}_{h}(S_{n})$ which contain $v_{0}$ but not as an end-vertex of the path is
		\[Z_{1}(a)=(1,\underbrace{2,\cdots,2}_{a},m,\underbrace{2,\cdots,2}_{h-1-a}), \quad0\leqslant a\leqslant h-2\] or
		\[Z_{2}(a)=(1, \underbrace{2, \cdots, 2}_{a}, m, \underbrace{2, \cdots, 2}_{h-2-a}, 1), \quad\left\{\begin{array}{l}0 \leqslant a \leqslant \frac{h}{2}-1, h \text { is even } \\ 0 \leqslant a \leqslant \frac{h-1}{2}-1, h \text { is odd }\end{array}\right.\] or
		\[Z_{3}(a)=(\underbrace{2, \cdots, 2}_{a}, m, \underbrace{2, \cdots, 2}_{h-a}), \quad\left\{\begin{array}{l}1 \leqslant a \leqslant \frac{h}{2}, h \text { is even } \\ 1 \leqslant a \leqslant \frac{h-1}{2}, h \text { is odd}\end{array}\right.\]
		then the number of paths in $\mathcal{P}_{h}(S_{n})$ whose images under $\Psi_{h}$ is $Z_{1}$ or $Z_{2}$ or $Z_{3}$ are
		\[| \Psi_{h}^{-1}(Z_{1}(a)) |=\left\{\begin{array}{cc}\left\{\begin{array}{cc}L_{a+1}\left(m-\sum_{i=1}^{h-(a+1)}L_{i}\right) & \text { if } 0 \leqslant a \leqslant \frac{h}{2}-1 \\ L_{a+1}\left(m-1-\sum_{i=1}^{h-(a+1)} L_{i}\right) & \text { if } \frac{h}{2} \leqslant a \leqslant h-2\end{array}\right\} \text { if } h \text { is even, } \\ \left\{\begin{array}{cc}L_{a+1}\left(m-\sum_{i=1}^{h-(a+1)} L_{i}\right) & \text { if } 0 \leqslant a<\frac{h-1}{2} \\ L_{a+1}\left(m-1-\sum_{i=1}^{h-(a+1)} L_{i}\right) & \text { if } \frac{h-1}{2} \leqslant a \leqslant h-2\end{array}\right\} \text { if } h \text { is odd, }\end{array}\right.\]
		
		$$\left|\Psi_{h}^{-1}\left(\mathrm{Z}_{2}(a)\right)\right|=\left\{\begin{array}{c}\left\{\begin{array}{c}L_{a+1} L_{h-a-1} \quad \text { if } 0 \leqslant a<\frac{h}{2}-1 \\ \frac{1}{2}\left(L_{a+1}-1\right) L_{a+1} \quad \text { if } a=\frac{h}{2}-1\end{array}\right\} \text { if } h \text { is even, } \\ \left\{L_{a+1} L_{h-a-1} \quad \text { if } 0 \leqslant a \leqslant \frac{h-1}{2}-1\right\} \text { if } h \text { is odd, }\end{array}\right.$$
		$$\begin{array}{l}\left|\Psi_{h}^{-1}\left(\mathrm{Z}_{3}(a)\right)\right|=\left\{\begin{array}{l}\left\{\begin{array}{cc}\left[m-\sum_{i=1}^{h-a} L_{i}\right]\left[m-1-\sum_{i=1}^{a} L_{i}\right] \text { if } 1 \leqslant a \leqslant \frac{h}{2}-1 \\ \frac{1}{2}\left[m-\sum_{i=1}^{a} L_{i}\right]\left[m-1-\sum_{i=1}^{a} L_{i}\right] \quad \text { if } a=\frac{h}{2}\end{array}\right\}\text { if } h \text { is even, }\\ \left\{\left[m-\sum_{i=1}^{h-a} L_{i}\right]\left[m-1-\sum_{i=1}^{a} L_{i}\right] \quad \text { if } 1 \leqslant a \leqslant \frac{h-1}{2}\right\} \text { if } h \text { is odd. }\end{array}\right.\end{array}	$$
	\end{lemma}	
	
	The following result shows that the $h$-order invariant $^hI_f$ of a starlike tree can be completely determined by its branches with length $\leq h$ and the function $f: \mathbb{Z}_{+}^{h+1} \rightarrow \mathbb{R}$.
	
	\begin{theorem}\label{t-2}
		Given a function $f:\mathbb{Z}_{+}^{h+1} \rightarrow \mathbb{R}$. $S_{n}$ is a starlike tree with $n$ vertices and $m$ is the degree of its root. Then the higher-order invariant $^{h}I_f$ of $S_{n}$ is completely determined by its branches whose length are not more than $h$, i.e.,
		\begin{align*}
			^{h}I_f(S_{n})=\overline{\lambda}+\overline{\mu}L_{h}
		\end{align*}
		where 
$\overline{\mu}=f(X_{1})-f(X_{2})+f(Y_{1})-f(Y_{2})$ is a real number determined by $X_{1}$, $X_{2}$, $Y_{1}$ and $Y_{2}$.
		More precisely, for $h>4$, the higher-order invariant $^{h}I_f$ of $S_{n}$ can be written by
		\begin{align*}
			^{h}I_f(S_{n})=&\lambda+\left[3f(Y_{1})-f(X_{2})-f(Y_{2})+L_1(f(Z_3(1))-f(Z_1(0))+f(Z_3(2))-f(Z_1(h-3)))\right.\\
			&\left.+L_2\left(f(Z_2(1))-f(Z_1(1))+f(Z_3(2))-f(Z_1(h-3))\right)\right.\\
			&\left.+(m-1)\left(f(Z_1(h-3))-f(Z_3(1))-f(Z_3(2))\right)\right]L_{h-2}\\
			&+\left[2f(Y_{1})-f(X_{2})-f(Y_{2})+L_1\left(f(Z_2(0))-f(Z_1(0))+f(Z_3(1))-f(Z_1(h-2))\right)\right.\\
			&\left.+(m-1)\left(f(Z_1(h-2))-f(Z_3(1))\right)\right]L_{h-1}\\&+\left[f(X_{1})-f(X_{2})+f(Y_{1})-f(Y_{2})\right]L_{h}
		\end{align*}
		where $\lambda$ is a real number. 
	\end{theorem}	
	
	\begin{proof}
		Let $\mathcal{P}_h$ be the set of all paths with length $h$ in $S_n$. From Lemma \ref{l-1} and $$^{h}I_f(G)=\sum_{\left(d_{0},d_{1},\cdots,d_{h}\right)}\omega\left(d_{0},d_{1},\cdots,d_{h}\right)f\left(d_{0},d_{1},\cdots,d_{h}\right),$$
		where $\omega\left(d_{0},d_{1},\cdots,d_{h}\right)$ denotes the number of paths with the degree sequence $\left(d_{0},d_{1},\cdots,d_{h}\right)$ in $G$, we can get
		\begin{equation*}
			\begin{aligned}
				^{0}I_f(S_{n})=&\sum_{v\in V(S_{n})}\omega(d(v))f(d(v))=f(m)+mf(1)+(n-m-1)f(2)\\
				=&f\left(m\right)+(n-1)f(2)+\left[f\left(1\right)-f\left(2\right)\right]m\\
				^{1}I_f(S_{n})=&\sum_{v_{0}v_{1}\in\mathcal{P}_{1}}\omega(d(v_0),d(v_1))f\left(d(v_{0}),d(v_{1})\right)\\
				=&\left|\Psi_{S_{n}}^{-1}\left(m,1\right)\right|f(m,1)+\left|\Psi_{S_{n}}^{-1}\left(m,2\right)\right|f(m,2)+\left|\Psi_{S_{n}}^{-1}\left(1,2\right)\right|f(1,2)+\left|\Psi_{S_{n}}^{-1}\left(2,2\right)\right|f(2,2)\\
				=&\left[f(m,2)+f(1,2)-2f(2,2)\right]m+(n-1)f(2,2)\\
				&+\left[f(m,1)-f(m,2)-f(1,2)+f(2,2)\right]L_{1}\\
				=&\overline{\lambda}(m,n)+\left[f(m,1)-f(m,2)+f(2,2)-f(1,2)\right]L_{1}
			\end{aligned}
		\end{equation*}
		
		\begin{equation*}
			\begin{aligned}
				^{2}I_f(S_{n})=&\sum_{v_{0}v_{1}v_{2}\in\mathcal{P}_{2}}\omega(d(v_0),d(v_1),d(v_2))f\left(d(v_{0}),d(v_{1}),d(v_{2})\right)\\
				=&\left|\Psi_{S_{n}}^{-1}\left(X_{1}\right)\right|f(X_{1})+\left|\Psi_{S_{n}}^{-1}\left(X_{2}\right)\right|f(X_{2})+\left|\Psi_{S_{n}}^{-1}\left(Y_{1}\right)\right|f(Y_{1})+\left|\Psi_{S_{n}}^{-1}\left(Y_{2}\right)\right|f(Y_{2})\\
				&+\left|\Psi_{S_{n}}^{-1}\left(Z_{1}(0)\right)\right|f(Z_{1}(0))+\left|\Psi_{S_{n}}^{-1}\left(Z_{2}(0)\right)\right|f(Z_{2}(0))+\left|\Psi_{S_{n}}^{-1}\left(Z_{3}(1)\right)\right|f(Z_{3}(1))\\
				=&L_{2}f(X_{1})+(m-L_{1}-L_{2})f(X_{2})+[(n-1)-L_{1}-2L_{2}-3(m-L_{1}-L_{2})]f(Y_{1})\\
				&+(m-L_{1}-L_{2})f(Y_{2})+L_{1}(m-L_{1})f(Z_{1}(0))+\frac{1}{2}L_{1}(L_{1}-1)f(Z_{2}(0))\\&+\frac{1}{2}(m-L_{1})(m-L_{1}-1)f(Z_{3}(1))\\
				=&\left[f(X_{1})-f(X_{2})+f(Y_{1})-f(Y_{2})\right]L_{2}+\left[2f(Y_{1})-f(X_{2})-f(Y_{2})\right]L_{1}\\
				&+mf(X_{2})+(n-1-3m)f(Y_{1})+mf(Y_{2})+L_{1}(m-L_{1})f(Z_{1}(0))\\
				&+\frac{1}{2}L_{1}(L_{1}-1)f(Z_{2}(0))+\frac{1}{2}(m-L_{1})(m-L_{1}-1)f(Z_{3}(1))\\
				=&\overline{\lambda}(n,m,L_{1})+\left[f(X_{1})-f(X_{2})+f(Y_{1})-f(Y_{2})\right]L_{2}
			\end{aligned}
		\end{equation*}
		
		\begin{equation*}
			\begin{aligned}
				^{3}I_f(S_{n})=&\sum_{v_{0}v_{1}v_{2}v_{3}\in\mathcal{P}_{3}}\omega(d(v_0),d(v_1),d(v_2),d(v_3))f\left(d(v_{0}),d(v_{1}),d(v_{2}),d(v_3)\right)\\
				=&\left|\Psi_{S_{n}}^{-1}\left(X_{1}\right)\right|f(X_{1})+\left|\Psi_{S_{n}}^{-1}\left(X_{2}\right)\right|f(X_{2})+\left|\Psi_{S_{n}}^{-1}\left(Y_{1}\right)\right|f(Y_{1})+\left|\Psi_{S_{n}}^{-1}\left(Y_{2}\right)\right|f(Y_{2})\\
				&+\left|\Psi_{S_{n}}^{-1}\left(Z_{1}(0)\right)\right|f(Z_{1}(0))+\left|\Psi_{S_{n}}^{-1}\left(Z_{1}(1)\right)\right|f(Z_{1}(1))+\left|\Psi_{S_{n}}^{-1}\left(Z_{2}(0)\right)\right|f(Z_{2}(0))\\
				&+\left|\Psi_{S_{n}}^{-1}\left(Z_{3}(1)\right)\right|f(Z_{3}(1))\\
				=&L_{3}f(X_{1})+(m-L_{1}-L_{2}-L_{3})f(X_{2})+(m-L_{1}-L_{2}-L_{3})f(Y_{2})\\
				&+[n-1-L_{1}-2L_{2}-3L_{3}-4(m-L_{1}-L_{2}-L_{3})]f(Y_{1})+L_{1}(m-L_{1}-L_{2})f(Z_{1}(0))\\&+L_{2}(m-1-L_{1})f(Z_{1}(1))+L_{1}L_{2}f(Z_{2}(0))+(m-L_{1}-L_{2})(m-L_{1}-1)f(Z_{3}(1))\\
			\end{aligned}
		\end{equation*}
		
		\begin{equation*}
			\begin{aligned}
				=&\left[f(X_{1})-f(X_{2})+f(Y_{1})-f(Y_{2})\right]L_{3}+\left[2f(Y_{1})-f(X_{2})-f(Y_{2})\right]L_{2}\\&+\left[3f(Y_{1})-f(X_{2})-f(Y_{2})\right]L_{1}+mf(X_{2})+(n-1-4m)f(Y_{1})+mf(Y_{2})\\&+L_{1}(m-L_{1}-L_{2})f(Z_{1}(0))\\&+L_{2}(m-1-L_{1})f(Z_{1}(1))+L_{1}L_{2}f(Z_{2}(0))+(m-L_{1}-L_{2})(m-L_{1}-1)f(Z_{3}(1))\\
				=&\overline{\lambda}(n,m,L_{1},L_{2})+\left[f(X_{1})-f(X_{2})+f(Y_{1})-f(Y_{2})\right]L_{3}
			\end{aligned}
		\end{equation*}
		
		\begin{equation*}
			\begin{aligned}
				^{4}I_f(S_{n})=&\sum_{v_{0}v_{1}v_{2}v_{3}v_{4}\in\mathcal{P}_{4}}\omega(d(v_0),d(v_1),d(v_2),d(v_3),d(v_4))f\left(d(v_{0}),d(v_{1}),d(v_{2}),d(v_3),d(v_4)\right)\\
				=&\left|\Psi_{S_{n}}^{-1}\left(X_{1}\right)\right|f(X_{1})+\left|\Psi_{S_{n}}^{-1}\left(X_{2}\right)\right|f(X_{2})+\left|\Psi_{S_{n}}^{-1}\left(Y_{1}\right)\right|f(Y_{1})+\left|\Psi_{S_{n}}^{-1}\left(Y_{2}\right)\right|f(Y_{2})\\&+\left|\Psi_{S_{n}}^{-1}\left(Z_{1}(0)\right)\right|f(Z_{1}(0))+\left|\Psi_{S_{n}}^{-1}\left(Z_{1}(1)\right)\right|f(Z_{1}(1))+\left|\Psi_{S_{n}}^{-1}\left(Z_{1}(1)\right)\right|f(Z_{1}(2))\\&+\left|\Psi_{S_{n}}^{-1}\left(Z_{2}(0)\right)\right|f(Z_{2}(0))+\left|\Psi_{S_{n}}^{-1}\left(Z_{2}(0)\right)\right|f(Z_{2}(1))+\left|\Psi_{S_{n}}^{-1}\left(Z_{3}(1)\right)\right|f(Z_{3}(1))\\&+\left|\Psi_{S_{n}}^{-1}\left(Z_{3}(2)\right)\right|f(Z_{3}(2))\\
				=&L_{4}f(X_{1})+\left(m-\sum_{i=1}^{4}L_{i}\right)f(X_{2})+[n-1-\sum_{i=1}^{4}iL_{i}-5\left(m-\sum_{i=1}^{4}L_{i}\right)]f(Y_{1})\\
				&+\left(m-\sum_{i=1}^{4}L_{i}\right)f(Y_{2})+L_{1}\left(m-\sum_{i=1}^{3}L_{i}\right)f(Z_{1}(0))+L_{2}\left(m-\sum_{i=1}^{2}L_{i}\right)f(Z_{1}(1))\\&+L_{3}\left(m-1-L_{1}\right)f(Z_{1}(2))+L_{1}L_{3}f(Z_{2}(0))+\frac{1}{2}L_{2}(L_{2}-1)f(Z_{2}(1))\\&+\left(m-\sum_{i=1}^{3}L_{i}\right)(m-L_{1}-1)f(Z_{3}(1))+\frac{1}{2}(m-L_{1}-L_{2})(m-1-L_{1}-L_{2})f(Z_{3}(2))\\
				=&\left[f(X_{1})-f(X_{2})+f(Y_{1})-f(Y_{2})\right]L_{4}+\left[2f(Y_{1})-f(X_{2})-f(Y_{2})\right]L_{3}\\&+\left[3f(Y_{1})-f(X_{2})-f(Y_{2})\right]L_{2}+(m-L_{1})f(X_{2})+(n-1-5m+4L_{1})f(Y_{1})\\&+(m-L_{1})f(Y_{2})+L_{1}\left(m-\sum_{i=1}^{3}L_{i}\right)f(Z_{1}(0))+L_{2}\left(m-\sum_{i=1}^{2}L_{i}\right)f(Z_{1}(1))\\&+L_{3}\left(m-1-L_{1}\right)f(Z_{1}(2))+L_{1}L_{3}f(Z_{2}(0))+\frac{1}{2}L_{2}(L_{2}-1)f(Z_{2}(1))\\&+\left(m-\sum_{i=1}^{3}L_{i}\right)(m-L_{1}-1)f(Z_{3}(1))+\frac{1}{2}(m-L_{1}-L_{2})(m-1-L_{1}-L_{2})f(Z_{3}(2))\\
				=&\overline{\lambda}(n,m,L_{1},L_{2},L_{3})+\left[f(X_{1})-f(X_{2})+f(Y_{1})-f(Y_{2})\right]L_{4}
			\end{aligned}
		\end{equation*}
		
		If $h\geqslant5$ is odd, then
		\begin{equation*}
			\begin{aligned}
				^{h}I_f(S_{n})=&\sum_{v_{0}v_{1}\cdots v_{h}\in\mathcal{P}_{h}}\omega(d(v_0),d(v_1),\cdots,d(v_h))f\left(d(v_0),d(v_1),\cdots,d(v_h)\right)\\
				=&\left|\Psi_{S_{n}}^{-1}\left(X_{1}\right)\right|f(X_1)+\left|\Psi_{S_{n}}^{-1}\left(X_{2}\right)\right|f(X_2)
				+\left|\Psi_{S_{n}}^{-1}\left(Y_{1}\right)\right|f(Y_1)+\left|\Psi_{S_{n}}^{-1}\left(Y_{2}\right)\right|f(Y_2)\\
				&+\sum_{a=0}^{h-2}\left|\Psi_{S_{n}}^{-1}\left(Z_{1}(a)\right)\right|f(Z_1(a))+\sum_{a=0}^{\frac{h-3}{2}}\left|\Psi_{S_{n}}^{-1}\left(Z_2(a)\right)\right|f(Z_2(a))
				+\sum_{a=1}^{\frac{h-1}{2}}\left|\Psi_{S_{n}}^{-1}\left(Z_3(a)\right)\right|f(Z_3(a)) \\
			\end{aligned}
		\end{equation*}
		
		\begin{equation*}
			\begin{aligned}				
				=&L_{h}f(X_1)+\left(m-\sum_{i=1}^{h}L_{i}\right)f(X_2)\\
				&+\left[(n-1)-\sum_{i=1}^{h}iL_{i}-(h+1)\left(m-\sum_{i=1}^{h}L_{i}\right)\right]f(Y_1)+\left(m-\sum_{i=1}^{h} L_{i}\right)f(Y_2)\\
				&+\left[\sum_{a=0}^{\frac{h-3}{2}}L_{a+1}\left(m-\sum_{i=1}^{h-(a+1)}L_{i}\right)+\sum_{a=\frac{h-1}{2}}^{h-2}L_{a+1}\left(m-1-\sum_{i=1}^{h-(a+1)} L_{i}\right)\right]f(Z_1(a))\\
				&+\left(\sum_{a=0}^{\frac{h-1}{2}-1}L_{a+1} L_{h-a-1}\right)f(Z_2(a))+\left[\sum_{a=1}^{\frac{h-1}{2}}\left(m-\sum_{i=1}^{h-a} L_{i}\right)\left(m-1-\sum_{i=1}^{a} L_{i}\right)\right]f(Z_3(a))\\
				=&\left[f(X_{1})-f(X_{2})+f(Y_{1})-f(Y_{2})\right]L_{h}+\left[2f(Y_{1})-f(X_{2})-f(Y_{2})\right]L_{h-1}\\
				&+\left[3f(Y_{1})-f(X_{2})-f(Y_{2})\right]L_{h-2}\\
				&+\left(m-\sum_{i=1}^{h-3}L_{i}\right)f(X_2)+\left[n-1-\sum_{i=1}^{h-3}iL_{i}-(h-1)\left(m-\sum_{i=1}^{h-3}L_{i}\right)\right]f(Y_{1})\\
				&+\left(m-\sum_{i=1}^{h-3}L_{i}\right)f(Y_{2})+L_1\left(m-\sum_{i=1}^{h-1}L_{i}\right)f(Z_1(0))+L_2\left(m-\sum_{i=1}^{h-2}L_{i}\right)f(Z_1(1))\\
				&+\sum_{a=2}^{\frac{h-3}{2}}L_{a+1}\left(m-\sum_{i=1}^{h-(a+1)}L_{i}\right)f(Z_1(a))+L_{h-1}\left(m-1-L_1\right)f(Z_1(h-2))\\
				&+L_{h-2}\left(m-1-L_1-L_2\right)f(Z_1(h-3))+\sum_{a=\frac{h-1}{2}}^{h-4}L_{a+1}\left(m-1-\sum_{i=1}^{h-(a+1)} L_{i}\right)f(Z_1(a))\\
				&+\left(\sum_{a=2}^{\frac{h-3}{2}}L_{a+1} L_{h-a-1}\right)f(Z_2(a))+L_{1}L_{h-1}f(Z_2(0))+L_{2}L_{h-2}f(Z_2(1))
				\\
				&+\left[\sum_{a=3}^{\frac{h-1}{2}}\left(m-\sum_{i=1}^{h-a} L_{i}\right)\left(m-1-\sum_{i=1}^{a} L_{i}\right)\right]f(Z_3(a))\\
				&+\left(m-\sum_{i=1}^{h-1} L_{i}\right)\left(m-1-L_{1}\right)f(Z_3(1))+\left(m-\sum_{i=1}^{h-2} L_{i}\right)\left(m-1-L_{1}-L_{2}\right)f(Z_3(2))\\
			\end{aligned}
		\end{equation*}
		\begin{equation*}
			\begin{aligned}
				=&\lambda\left[n,m,h,L_{1},\cdots,L_{h-3},f(X_{1}),f(X_{2}),f(Y_{1}),f(Y_{2}),f(Z_1(a)),f(Z_2(a)),f(Z_3(a))\right]\\
				&+\left[f(X_{1})-f(X_{2})+f(Y_{1})-f(Y_{2})\right]L_{h}\\
				&+\left[2f(Y_{1})-f(X_{2})-f(Y_{2})+L_1\left(f(Z_2(0))-f(Z_1(0))+f(Z_3(1))-f(Z_1(h-2))\right)\right.\\
				&\left.+(m-1)\left(f(Z_1(h-2))-f(Z_3(1))\right)\right]L_{h-1}\\&+\left[3f(Y_{1})-f(X_{2})-f(Y_{2})+L_1(f(Z_3(1))-f(Z_1(0))+f(Z_3(2))-f(Z_1(h-3)))\right.\\
				&\left.+L_2\left(f(Z_2(1))-f(Z_1(1))+f(Z_3(2))-f(Z_1(h-3))\right)\right.\\
				&\left.+(m-1)\left(f(Z_1(h-3))-f(Z_3(1))-f(Z_3(2))\right)\right]L_{h-2}\\
			\end{aligned}
		\end{equation*}
		
		If $h\geqslant6$ is even, then
		\begin{equation*}
			\begin{aligned}
				^{h}I_f(S_{n})=&\sum_{v_{0}v_{1}\cdots v_{h}\in\mathcal{P}_{h}}\omega(d(v_0),d(v_1),\cdots,d(v_h))f\left(d(v_0),d(v_1),\cdots,d(v_h)\right)\\
				=&\left|\Psi_{S_{n}}^{-1}\left(X_{1}\right)\right|f(X_1)+\left|\Psi_{S_{n}}^{-1}\left(X_{2}\right)\right|f(X_2)+\left|\Psi_{S_{n}}^{-1}\left(Y_{1}\right)\right|f(Y_1)+\left|\Psi_{S_{n}}^{-1}\left(Y_{2}\right)\right|f(Y_2)\\
				&+\sum_{a=0}^{h-2}\left|\Psi_{S_{n}}^{-1}\left(Z_{1}(a)\right)\right|f(Z_1(a))+\sum_{a=0}^{\frac{h}{2}-1}\left|\Psi_{S_{n}}^{-1}\left(Z_2(a)\right)\right|f(Z_2(a))+\sum_{a=1}^{\frac{h}{2}}\left|\Psi_{S_{n}}^{-1}\left(Z_3(a)\right)\right|f(Z_3(a)) \\
				=&L_{h}f(X_1)+\left(m-\sum_{i=1}^{h}L_{i}\right)f(X_2)+\left[(n-1)-\sum_{i=1}^{h}iL_{i}-(h+1)\left(m-\sum_{i=1}^{h}L_{i}\right)\right]f(Y_1)\\
				&+\left(m-\sum_{i=1}^{h} L_{i}\right)f(Y_2)+\sum_{a=0}^{\frac{h-4}{2}}L_{a+1} L_{h-a-1}f(Z_2(a))+\frac{1}{2}L_{\frac{h}{2}}\left(L_{\frac{h}{2}}-1\right)f(Z_2(\frac{h}{2}-1))\\
				&+\left[\sum_{a=0}^{\frac{h-2}{2}}L_{a+1}\left(m-\sum_{i=1}^{h-(a+1)}L_{i}\right)+\sum_{a=\frac{h}{2}}^{h-2}L_{a+1}\left(m-1-\sum_{i=1}^{h-(a+1)} L_{i}\right)\right]f(Z_1(a))\\&+\left[\sum_{a=1}^{\frac{h-2}{2}}\left(m-\sum_{i=1}^{h-a} L_{i}\right)\left(m-1-\sum_{i=1}^{a} L_{i}\right)\right]f(Z_3(a))\\&+\frac{1}{2}\left(m-\sum_{i=1}^{\frac{h}{2}} L_{i}\right)\left(m-1-\sum_{i=1}^{\frac{h}{2}} L_{i}\right)f(Z_3(\frac{h}{2}))\\
				=&\left[f(X_{1})-f(X_{2})+f(Y_{1})-f(Y_{2})\right]L_{h}+\left[2f(Y_{1})-f(X_{2})-f(Y_{2})\right]L_{h-1}\\
				&+\left[3f(Y_{1})-f(X_{2})-f(Y_{2})\right]L_{h-2}+\left(m-\sum_{i=1}^{h-3}L_{i}\right)f(X_2)\\
				&+\left[n-1-\sum_{i=1}^{h-3}iL_{i}-(h-1)\left(m-\sum_{i=1}^{h-3}L_{i}\right)\right]f(Y_{1})+\left(m-\sum_{i=1}^{h-3}L_{i}\right)f(Y_{2})\\
				&+L_1\left(m-\sum_{i=1}^{h-1}L_{i}\right)f(Z_1(0))+L_2\left(m-\sum_{i=1}^{h-2}L_{i}\right)f(Z_1(1))\\
				&+\sum_{a=2}^{\frac{h-2}{2}}L_{a+1}\left(m-\sum_{i=1}^{h-(a+1)}L_{i}\right)f(Z_1(a))+L_{h-1}\left(m-1-L_1\right)f(Z_1(h-2))\\
				&+L_{h-2}\left(m-1-L_1-L_2\right)f(Z_1(h-3))+\sum_{a=\frac{h}{2}}^{h-4}L_{a+1}\left(m-1-\sum_{i=1}^{h-(a+1)} L_{i}\right)f(Z_1(a))\\
				&+\left(\sum_{a=3}^{\frac{h-4}{2}}L_{a+1} L_{h-a-1}\right)f(Z_2(a))+\frac{1}{2}L_{\frac{h}{2}}\left(L_{\frac{h}{2}}-1\right)f(Z_2(\frac{h}{2}-1))+L_{1}L_{h-1}f(Z_2(0))\\
				&+L_{2}L_{h-2}f(Z_2(1))+\left[\sum_{a=3}^{\frac{h-2}{2}}\left(m-\sum_{i=1}^{h-a} L_{i}\right)\left(m-1-\sum_{i=1}^{a} L_{i}\right)\right]f(Z_3(a))\\
			\end{aligned}
		\end{equation*}
		
		\begin{equation*}
			\begin{aligned}&+\left(m-\sum_{i=1}^{h-1} L_{i}\right)\left(m-1-L_{1}\right)f(Z_3(1))+\left(m-\sum_{i=1}^{h-2} L_{i}\right)\left(m-1-L_{1}-L_{2}\right)f(Z_3(2))\\
				&+\frac{1}{2}\left(m-\sum_{i=1}^{\frac{h}{2}} L_{i}\right)\left(m-1-\sum_{i=1}^{\frac{h}{2}} L_{i}\right)f(Z_3(\frac{h}{2}))\\
				=&\lambda\left[n,m,h,L_{1},\cdots,L_{h-3},f(X_{1}),f(X_{2}),f(Y_{1}),f(Y_{2}),f(Z_1(a)),f(Z_2(a)),f(Z_3(a))\right]\\
				&+\left[f(X_{1})-f(X_{2})+f(Y_{1})-f(Y_{2})\right]L_{h}\\
				&+\left[2f(Y_{1})-f(X_{2})-f(Y_{2})+L_1\left(f(Z_2(0))-f(Z_1(0))+f(Z_3(1))-f(Z_1(h-2))\right)\right.\\
				&\left.+(m-1)\left(f(Z_1(h-2))-f(Z_3(1))\right)\right]L_{h-1}\\&+\left[3f(Y_{1})-f(X_{2})-f(Y_{2})+L_1(f(Z_3(1))-f(Z_1(0))+f(Z_3(2))-f(Z_1(h-3)))\right.\\
				&\left.+L_2\left(f(Z_2(1))-f(Z_1(1))+f(Z_3(2))-f(Z_1(h-3))\right)\right.\\
				&\left.+(m-1)\left(f(Z_1(h-3))-f(Z_3(1))-f(Z_3(2))\right)\right]L_{h-2}\\
			\end{aligned}
		\end{equation*}
	\end{proof}
	
	Specially, if $f\equiv 1$, then we can get
	
	\begin{corollary}\label{c-3}\cite{cai2024}
		If $S_{n}$ is a starlike tree with $n$ vertices and the degree $m>2$ of its root, then the number of paths with length $h$ in $S_n$ is
		$$P_{h}(S_{n})=|\mathcal{P}_{h}(S_{n})|=\lambda(n,m,h,L_{1},\cdots,L_{h-3})+\mu(n,m,h,L_{1},\cdots,L_{h-3})L_{h-2}$$
		where $\lambda(n,m,h,L_{1},\cdots,L_{h-3})$ is a real number determined by the values of $n,m,h,L_{1},\cdots,L_{h-3}$ and $\mu(n,m,h,L_{1},\cdots,L_{h-3})=2-m, h>2$.
	\end{corollary}
	
	If $f\left(x_{0},x_{1},\cdots,x_{h}\right)=\frac{1}{\sqrt{x_{0}x_{1}\cdots x_{h}}}$, then we get
	
	\begin{corollary}\label{c-4}\cite{rada2002higher}
		If $S_{n}$ is a starlike tree with $n$ vertices and the degree $m>2$ of its root, then the $h$-connectivity index of $S_n$ is
		$$^{h}\chi(S_{n})=\lambda(n,m,h,L_{1},\cdots,L_{h-1})+\mu(n,m)L_{h}$$
		where $\lambda(n,m,h,L_{1},\cdots,L_{h-1})$ is a real number determined by the values of $n,m,h,L_{1},\cdots,L_{h-1}$ and $\mu(h,m)=\frac{1}{\sqrt{2^{h-1}m}}-\frac{1}{\sqrt{2^hm}}+\frac{1}{\sqrt{2^{h+1}}}-\frac{1}{\sqrt{2^{h}}}$ is a real number determined by $h$ and $m$.
	\end{corollary}
	
	Next, we consider the generalized starlike trees.
	
	A coalescence of two vertex disjoint graphs $G_{1}=(V(G_{1}),E(G_{1}))$ and $G_{2}=(V(G_{2}),E(G_{2}))$ with respect to vertex $u \in G_{1}$ and vertex $v \in G_{2}$, denoted by $G_{1}(u=v)G_{2}$, is obtained from the union of $G_{1}$ and $G_{2} $ by identifying vertices $u$ and $v$.
	
	$G_{n} = K_{n_{1}}(u_{0}=v_{0})S_{n_{2}}$ is called a generalized starlike tree, where $u_{0}\in K_{n_{1}}$ and $v_{0}$ is the root of $S_{n_{2}}$, $K_{n_{1}}$ and $S_{n_{2}}$ are the complete graph with $n_{1}>1$ vertices and a starlike tree with $n_{2}$ vertices. More specifically, renaming both $u_{0}$ and $v_{0}$ with $x_{0}$, where $x_{0} \notin V(K_{n_{1}})\cup V(S_{n_{2}})$, the vertex set of $G_{n}$ is $\left(V(K_{n_{1}}) \backslash\{u_{0}\}\right) \cup\left(V(S_{n_{2}}) \backslash\{v_{0}\}\right) \cup\left\{x_{0}\right\}$.  A branch of length $t$ in $S_{n_{2}}$ is also called a branch of the generalized starlike tree $G_{n}$, and the number of branches with length $t$ is still denoted by $L_{t}$.
	
	In \cite{cai2024}, we gave the method of counting the number of paths with length $h$ in a generalized starlike tree $G_{n} = K_{n_{1}}(u_{0}=v_{0})S_{n_{2}}$. Here, the calculation of the $h$-order invariant of a generalized starlike tree $^{h}I_f(G_{n})$ depends on the degree sequence of the paths with length $h$. Using the calculation method in Lemma \ref{l-1}, we can easily obtain the number of paths with a given degree sequences.
	
	\begin{lemma}\label{l-5}
		Let $G_{n} = K_{n_{1}}(u_{0}=v_{0})S_{n_{2}}$ and $m$ the degree of root in $S_{n_{2}}$. $\mathcal{P}_h(G_n)$ is the set of all paths with length $h$ in $G_{n}$. Then the number $|\mathcal{P}_h(G_n)|$ of paths with length $h$ in $G_n$ can be counted by the possible image of the degree sequence $(d_{v_{0}},d_{v_{1}},\cdots,d_{v_{h}})$ under $\Psi_{h}$:
		
		\textbf{Type 1}
		The possible image under $\Psi_h$ of all paths in $\mathcal{P}_{h}(G_{n})$  containing only vertices in $K_{n_{1}}$ can be divided into three types according to the degree sequence of the paths:
		$$\mathrm{U}=(m+n_{1}-1, \underbrace{n_{1}-1,\cdots,n_{1}-1}_{h})$$
		$$\mathrm{V}=( \underbrace{n_{1}-1,\cdots,n_{1}-1}_{h+1})$$
		$$\mathrm{W(a)}=(\underbrace{n_{1}-1,\cdots,n_{1}-1}_{a},  m+n_{1}-1,\underbrace{n_{1}-1,\cdots,n_{1}-1}_{h-a}),\quad\left\{\begin{array}{l}1 \leqslant a \leqslant \frac{h}{2}, h \text { is even } \\ 1 \leqslant a \leqslant \frac{h-1}{2}, h \text { is odd}\end{array}\right.$$
		Then the number of paths in $\mathcal{P}_{h}(G_{n})$ whose degree sequence under $\Psi_{G_{n}}$ is $\mathrm{U}$, $\mathrm{V}$ or $\mathrm{W(a)}$ is
		$$\left|\Psi_{h}^{-1}\left(\mathrm{U}\right)\right|=\prod_{i=1}^{h}(n_1-i),\;
		\left|\Psi_{h}^{-1}\left(\mathrm{V}\right)\right|=\frac{1}{2}\prod_{i=1}^{h+1}(n_1-i),$$
		$$\begin{array}{l}\left|\Psi_{h}^{-1}\left(\mathrm{W}(a)\right)\right|=\left\{\begin{array}{l}\left\{\begin{array}{cc}\prod\limits_{i=1}^{a}(n_1-i)\prod\limits_{i=1}^{h-a}(n_1-a-i) \text { if } 1 \leqslant a \leqslant \frac{h}{2}-1 \\\dfrac{1}{2}\prod\limits_{i=1}^{a}(n_1-i)\prod\limits_{i=1}^{h-a}(n_1-a-i) \quad \text { if } a=\frac{h}{2}\end{array}\right\}\text { if } h \text { is even, }\\ \left\{\prod\limits_{i=1}^{a}(n_1-i)\prod\limits_{i=1}^{h-a}(n_1-a-i) \quad \text { if } 1 \leqslant a \leqslant \frac{h-1}{2}\right\} \text { if } h \text { is odd. }\end{array}\right.\end{array}.	$$		
		
		\textbf{Type 2}
		The possible image under $\Psi_{h}$ of all paths in $\mathcal{P}_{h}(G_{n})$  containing only vertices in $S_{n_{2}}$ can be divided into three types according to the degree sequence of the paths:\\
		$$\mathrm{X'_{1}}=(m+n_1-1,\underbrace{2,\cdots,2}_{h-1},1),\mathrm{X'_{2}}=(m+n_1-1,\underbrace{2,\cdots,2}_{h})$$
		$$\mathrm{Y'_{1}}=(\underbrace{2,\cdots,2}_{h+1}),\mathrm{Y'_{2}}=(1,\underbrace{2,\cdots,2}_{h})$$
		$$\mathrm{Z'_{1}(a)}=(1,\underbrace{2,\cdots,2}_{a},m+n_1-1,\underbrace{2,\cdots,2}_{h-1-a}), \quad0\leqslant a\leqslant h-2$$
		$$\mathrm{Z'_{2}(a)}=(1, \underbrace{2, \cdots, 2}_{a}, m+n_1-1, \underbrace{2, \cdots, 2}_{h-2-a}, 1), \quad\left\{\begin{array}{l}0 \leqslant a \leqslant \frac{h}{2}-1, h \text { is even } \\ 0 \leqslant a \leqslant \frac{h-1}{2}-1, h \text { is odd }\end{array}\right.$$
		$$\mathrm{Z'_{3}(a)}=(\underbrace{2, \cdots, 2}_{a}, m+n_1-1, \underbrace{2, \cdots, 2}_{h-a}), \quad\left\{\begin{array}{l}1 \leqslant a \leqslant \frac{h}{2}, h \text { is even } \\ 1 \leqslant a \leqslant \frac{h-1}{2}, h \text { is odd}\end{array}\right.$$
		Then the number of paths in $\mathcal{P}_h(G_n)$ whose degree sequence under $\Psi_{G_{n}}$ is $\mathrm{X'_{1}}$, $\mathrm{X'_{2}}$, $\mathrm{Y'_{1}}$, $\mathrm{Y'_{2}}$, $\mathrm{Z'_{1}(a)}$, $\mathrm{Z'_{2}(a)}$ or $\mathrm{Z'_{3}(a)}$ is the same as the result in Lemma \ref{l-1}, i.e., $|\Psi_{h}^{-1}(\mathrm{X'_{i}})|=|\Psi_{h}^{-1}(X_{i})|$, $|\Psi_{h}^{-1}(\mathrm{Y'_{i}})|=|\Psi_{h}^{-1}(Y_{i})|$, $i=1,2$; $|\Psi_{h}^{-1}(\mathrm{Z'_{j}(a)})|=|\Psi_{h}^{-1}(Z_{j})(a)|$, $j=1,2,3$.
		
		\textbf{Type 3}
		The possible image under $\Psi_{G_{n}}$ of all paths in $\mathcal{P}_{h}(G_{n})$  connected by $x_{0}$ between a path of length $a$ in $K_{n_{1}}$ and a path of length $h-a$ in $S_{n_{2}}$ can be divided into two types according to the degree sequence of the paths:
		$$\mathrm{M_1(a)}=(\underbrace{n_{1}-1,\cdots,n_{1}-1}_{a}, m+n_{1}-1,\underbrace{2\cdots,2,2}_{h-a}),\; 1\leqslant a\leqslant h-1$$ $$\mathrm{M_2(a)}=(\underbrace{n_{1}-1,\cdots,n_{1}-1}_{a},  m+n_{1}-1,\underbrace{2\cdots,2,2}_{h-a-1},1),1\leqslant a\leqslant h-1$$
		Then the number of paths in $\mathcal{P}_{h}(G_{n})$ whose degree sequence under $\Psi_{h}$ is $\mathrm{M_1(a)}$ or $\mathrm{M_2(a)}$ is
		$$\left|\Psi_{h}^{-1}\left(\mathrm{M_1(a)}\right)\right|=\prod_{i=1}^{a}\left(n_{1}-i\right)\left(m-\sum_{i=1}^{h-a}L_{i}\right),$$
		$$\left|\Psi_{h}^{-1}\left(\mathrm{M_2(a)}\right)\right|=\prod_{i=1}^{a}\left(n_{1}-i\right)L_{h-a}.$$	
	\end{lemma}
	
	In the following, we will show that $h$-order invariant $^hI_f$ of a generalized starlike tree is completely determined by its branches with length $\leq h$ and the function $f:\mathbb{Z}_{+}^{h+1} \rightarrow \mathbb{R}$.

	\begin{theorem}\label{t-6}
		Given a function $f:\mathbb{Z}_{+}^{h+1} \rightarrow \mathbb{R}$. $G_{n}= K_{n_{1}}(u_{0}=v_{0})S_{n_{2}}$ is a generalized starlike tree with $n=n_2+n_1-1$ vertices and $m$ is the degree of root in $S_{n_2}$. Then the higher-order invariant $^{h}I_f$ of $G_{n}$ is completely determined by its branches whose length are not more than $h$, i.e.,
		\begin{equation*}
			\begin{aligned} ^{h}I_f(G_{n})=&\overline{\lambda^{\prime}}\left[n_1,n_2,m,h,L_{1},\cdots,L_{h-1},f(\mathrm{X'_{1}}),f(\mathrm{X'_{2}}),f(\mathrm{Y'_{1}}),f(\mathrm{Y'_{2}}),f(\mathrm{Z'_{1}(a)}),f(\mathrm{Z'_{2}(a)}),\right.\\
				&\quad\left.f(\mathrm{Z'_{3}(a)}),f(\mathrm{M_1(a)}), f(\mathrm{M_2(a)}),f(\mathrm{U}),f(\mathrm{V}),f(\mathrm{W(a)})\right]+\overline{\mu^{\prime}}L_{h},\\
where \;  \overline{\mu^{\prime}}=&f(\mathrm{X'_{1}})-f(\mathrm{X'_{2}})+f(\mathrm{Y'_{1}})-f(\mathrm{Y'_{2}}).
			\end{aligned}
		\end{equation*}
		
	\end{theorem}
	
	\begin{proof}
		Let $\mathcal{P}_h$ be the set of all paths with length $h$ in $G_n$. 
		
		\begin{equation*}
			\begin{aligned}
				^{0}I_f(G_{n})=&\sum_{v\in V(G_{n})} \omega(d(v))f\left(d_{v}\right)\\=&f(m+n_{1}-1)+mf(1)+(n_{2}-m-1)f(2)+(n_{1}-1)f(n_{1}-1)\\
				=&\left[f\left(1\right)-f\left(2\right)\right]m+f\left(m+n_{1}-1\right)+(n_{1}-1)f(n_{1}-1)+(n_2-1)f(2)\\
			\end{aligned}
		\end{equation*}
		\begin{equation*}
			\begin{aligned}
				^{1}I_f(G_{n})=&\sum_{v_{0}v_{1}\in\mathcal{P}_{1}}\omega(d(v_0),d(v_1))f\left(d(v_{0}),d(v_{1})\right)\\
				=&\left|\Psi_{G_{n}}^{-1}\left(m+n_{1}-1,n_{1}-1\right)\right|f(m+n_{1}-1,n_{1}-1)+\left|\Psi_{G_{n}}^{-1}\left(n_{1}-1,n_{1}-1\right)\right|f(n_{1}-1,n_{1}-1)\\&+\left|\Psi_{G_{n}}^{-1}\left(m+n_1-1,1\right)\right|f(m+n_{1}-1,1)+\left|\Psi_{G_{n}}^{-1}\left(m+n_{1}-1,2\right)\right|f(m+n_{1}-1,2)\\&+\left|\Psi_{G_{n}}^{-1}\left(1,2\right)\right|f(1,2)+\left|\Psi_{G_{n}}^{-1}\left(2,2\right)\right|f(2,2)\\
				=&(n_{1}-1)f(m+n_{1}-1,n_{1}-1)+\frac{(n_{1}-1)(n_{1}-2)}{2}f(n_{1}-1,n_{1}-1)+L_{1}f(m+n_{1}-1,1)\\&+(m-L_{1})f(m+n_{1}-1,2)+(m-L_{1})f(1,2)+(n-n_{1}-2m+L_{1})f(2,2)\\
				=&\overline{\lambda^{\prime}}\left[m,n_{1},n_{2},f(m+n_{1}-1,1),f(m+n_{1}-1,2),f(1,2),f(2,2)\right]\\&+\left[f(m+n_{1}-1,1)-f(m+n_{1}-1,2)-f(1,2)+f(2,2)\right]L_{1}\\
			\end{aligned}
		\end{equation*}
		
		\begin{equation*}
			\begin{aligned}	
				^{2}I_f(G_{n})=&\sum_{v_{0}v_{1}v_{2}\in\mathcal{P}_{2}}\omega(d(v_0),d(v_1),d(v_2))f\left(d(v_{0}),d(v_{1}),d(v_{2})\right)\\
				=&\left|\Psi_{G_{n}}^{-1}\left(m+n_{1}-1,2,1\right)\right|f(m+n_{1}-1,2,1)+\left|\Psi_{G_{n}}^{-1}\left(m+n_{1}-1,2,2\right)\right|f(m+n_{1}-1,2,2)\\
				&+\left|\Psi_{G_{n}}^{-1}\left(2,2,2\right)\right|f(2,2,2)+\left|\Psi_{G_{n}}^{-1}\left(1,m+n_{1}-1,2\right)\right|f(1,m+n_{1}-1,2)\\
				&+\left|\Psi_{G_{n}}^{-1}\left(1,2,2\right)\right|f(1,2,2)+\left|\Psi_{G_{n}}^{-1}\left(1,m+n_{1}-1,1\right)\right|f(1,m+n_{1}-1,1)\\
				&+\left|\Psi_{G_{n}}^{-1}\left(2,m+n_{1}-1,2\right)\right|f(2,m+n_{1}-1,2)\\
				&+\left|\Psi^{-1}_{G_{n}}\left(n_{1}-1,n_{1}-1,n_{1}-1\right)\right|f(n_{1}-1,n_{1}-1,n_{1}-1)\\
				&+\left|\Psi^{-1}_{G_{n}}\left(m+n_{1}-1,n_{1}-1,n_{1}-1\right)\right|f(m+n_{1}-1,n_{1}-1,n_{1}-1)\\
				&+\left|\Psi_{G_{n}}^{-1}\left(n_{1}-1,m+n_{1}-1,1\right)\right|f(n_{1}-1,m+n_{1}-1,1)\\
				&+\left|\Psi_{G_{n}}^{-1}\left(n_{1}-1,m+n_{1}-1,2\right)\right|f(n_{1}-1,m+n_{1}-1,2)\\
				=&L_{2}f(m+n_{1}-1,2,1)+(n-n_{1}-3m+2L_{1}+L_{2})f(2,2,2)+(m-L_{1}-L_{2})f(1,2,2)\\
				&+(m-L_{1}-L_{2})f(m+n_{1}-1,2,2)+L_{1}(m-L_{1})f(1,m+n_{1}-1,2)\\
				&+\frac{L_{1}}{2}(L_{1}-1)f(1,m+n_{1}-1,1)+\frac{1}{2}(m-L_{1})(m-L_{1}-1)f(2,m+n_{1}-1,2)\\
				&+\frac{1}{2}\prod_{i=1}^{3}(n_{1}-i)f(n_{1}-1,n_{1}-1,n_{1}-1)+(n_{1}-1)(n_{1}-2)f(m+n_{1}-1,n_{1}-1,n_{1}-1)\\
				&+(n_{1}-1)L_{1}f(n_{1}-1,m+n_{1}-1,1)+(n_{1}-1)(m-L_{1})f(n_{1}-1,2,2)\\
				=&\overline{\lambda^{\prime}}[m,n_{1},n_2,L_{1},f(m+n_{1}-1,2,1),f(m+n_{1}-1,2,2),f(2,2,2),f(1,2,2),\\
				&f(1,m+n_{1}-1,2),f(1,m+n_{1}-1,1),f(2,m+n_{1}-1,2),f(n_{1}-1,n_{1}-1,n_{1}-1),\\
				&f(n_{1}-1,m+n_{1}-1,1),f(m+n_{1}-1,n_{1}-1,n_{1}-1),f(n_{1}-1,m+n_{1}-1,2)]\\
				&+\left[f(m+n_{1}-1,2,1)-f(m+n_{1}-1,2,2)+f(2,2,2)-f(1,2,2)\right]L_{2}\\
			\end{aligned}
		\end{equation*}
		
		\begin{equation*}
			\begin{aligned}	
				^{3}I_f(G_{n})=&\sum_{v_{0}v_{1}v_{2}v_{3}\in\mathcal{P}_{3}}\omega(d(v_0),d(v_1),d(v_2),d(v_3))f\left(d(v_{0}),d(v_{1}),d(v_{2},d(v_3))\right)\\
				=&\left|\Psi_{G_{n}}^{-1}\left(\mathrm{U}\right)\right|f(\mathrm{U})+\left|\Psi_{G_{n}}^{-1}\left(\mathrm{V}\right)\right|f(\mathrm{V})
				+\left|\Psi_{G_{n}}^{-1}\left(\mathrm{W(1)}\right)\right|f(\mathrm{W(1)})+^{3}I_f(S_{n_2})\\
				&+\left|\Psi^{-1}_{G_{n}}\left(\mathrm{M_1(1)}\right)\right|f(\mathrm{M_1(1)})+\left|\Psi^{-1}_{G_{n}}\left(\mathrm{M_1(2)}\right)\right|f(\mathrm{M_1(2)})
				+\left|\Psi_{G_{n}}^{-1}\left(\mathrm{M_2(1)}\right)\right|f(\mathrm{M_2(1)})\\
				&+\left|\Psi_{G_{n}}^{-1}\left(\mathrm{M_2(2)}\right)\right|f(\mathrm{M_2(2)})\\  	
				=&\prod_{i=1}^{3}(n_1-i)f(\mathrm{U})+\frac{1}{2}\prod_{i=1}^{4}(n_1-i)f(\mathrm{V})+\prod_{i=1}^{3}(n_1-i)f(\mathrm{W(1)})\\
				&+\left[f(\mathrm{X'_{1}})-f(\mathrm{X'_{2}})+f(\mathrm{Y'_{1}})-f(\mathrm{Y'_{2}})\right]L_{3}+\left[2f(\mathrm{Y'_{1}})-f(\mathrm{X'_{2}})-f(\mathrm{Y'_{2}})\right]L_{2}\\
				&+\left[3f(\mathrm{Y'_{1}})-f(\mathrm{X'_{2}})-f(\mathrm{Y'_{2}})\right]L_{1}+mf(\mathrm{X'_{2}})+(n_2-1-4m)f(\mathrm{Y'_{1}})+mf(\mathrm{Y'_{2}})\\
				&+L_{1}(m-L_{1}-L_{2})f(\mathrm{Z'_{1}(0)})+L_{2}(m-1-L_{1})f(\mathrm{Z'_{1}(1)})+L_{1}L_{2}f(\mathrm{Z'_{2}(0)})\\
				&+(m-L_{1}-L_{2})(m-L_{1}-1)f(\mathrm{Z'_{3}(1)})+\left(n_{1}-1\right)\left(m-L_{1}-L_{2}\right)f(\mathrm{M_1(1)})\\
				&+\left(n_{1}-1\right)\left(n_{1}-2\right)\left(m-L_{1}\right)f(\mathrm{M_1(2)})\\&+\left(n_{1}-1\right)L_{2}f(\mathrm{M_2(1)})+\left(n_{1}-1\right)\left(n_{1}-2\right)L_{1}f(\mathrm{M_2(2)})\\
				=&\left[f(\mathrm{X'_{1}})-f(\mathrm{X'_{2}})+f(\mathrm{Y'_{1}})-f(\mathrm{Y'_{2}})\right]L_{3}\\&+\left[2f(\mathrm{Y'_{1}})-f(\mathrm{X'_{2}})-f(\mathrm{Y'_{2}})+(n_1-1)\left(f(\mathrm{M_2(1)})-f(\mathrm{M_1(1)})\right)\right]L_{2}\\&+\left[3f(\mathrm{Y'_{1}})-f(\mathrm{X'_{2}})-f(\mathrm{Y'_{2}})+\left(n_{1}-1\right)\left(n_{1}-2\right)\left(f(\mathrm{M_2(2)})-f(\mathrm{M_1(2)})\right)\right.\\
				&\quad\left.-\left(n_{1}-1\right)f(\mathrm{M_1(1)})\right]L_{1}+mf(\mathrm{X'_{2}})+(n_2-1-4m)f(\mathrm{Y'_{1}})+mf(\mathrm{Y'_{2}})\\&+L_{1}(m-L_{1}-L_{2})f(\mathrm{Z'_{1}(0)})+L_{2}(m-1-L_{1})f(\mathrm{Z'_{1}(1)})+L_{1}L_{2}f(\mathrm{Z'_{2}(0)})\\&+(m-L_{1}-L_{2})(m-L_{1}-1)f(\mathrm{Z'_{3}(1)})+m\left(n_{1}-1\right)f(\mathrm{M_1(1)})\\&+m\left(n_{1}-1\right)\left(n_{1}-2\right)f(\mathrm{M_1(2)})\\
				=&\overline{\lambda^{\prime}}\left[n_1,n_2,m,L_{1},L_2,f(\mathrm{X'_{1}}),f(\mathrm{X'_{2}}),f(\mathrm{Y'_{1}}),f(\mathrm{Y'_{2}}),f(\mathrm{Z'_{1}(0)}),f(\mathrm{Z'_{1}(1)}),f(\mathrm{Z'_{2}(0)}),f(\mathrm{Z'_{3}(1)})\right.\\
				&\quad\left. f(\mathrm{M_1(1)}),f(\mathrm{M_1(2)}),f(\mathrm{U}),f(\mathrm{V}),f(\mathrm{W(1)})\right]+\left[f(\mathrm{X'_{1}})-f(\mathrm{X'_{2}})+f(\mathrm{Y'_{1}})-f(\mathrm{Y'_{2}})\right]L_{3}\\
			\end{aligned}
		\end{equation*}
		
		\begin{equation*}
			\begin{aligned}	
				^{4}I_f(G_{n})=&\sum_{v_{0}v_{1}v_{2}v_{3}v_{4}\in\mathcal{P}_{4}}\omega(d(v_0),d(v_1),d(v_2),d(v_3),d(v_4))f\left(d(v_{0}),d(v_{1}),d(v_{2},d(v_3),d(v_4))\right)\\
				=&\left|\Psi_{G_{n}}^{-1}\left(\mathrm{U}\right)\right|f(\mathrm{U})+\left|\Psi_{G_{n}}^{-1}\left(\mathrm{V}\right)\right|f(\mathrm{V})+\left|\Psi_{G_{n}}^{-1}\left(\mathrm{W(1)}\right)\right|f(\mathrm{W(1)})+\left|\Psi_{G_{n}}^{-1}\left(\mathrm{W(2)}\right)\right|f(\mathrm{W(2)})\\&+^{4}I_f(S_{n_2})+\left|\Psi^{-1}_{G_{n}}\left(\mathrm{M_1(1)}\right)\right|f(\mathrm{M_1(1)})+\left|\Psi^{-1}_{G_{n}}\left(\mathrm{M_1(2)}\right)\right|f(\mathrm{M_1(2)})+\left|\Psi^{-1}_{G_{n}}\left(\mathrm{M_1(3)}\right)\right|f(\mathrm{M_1(3)})\\&+\left|\Psi_{G_{n}}^{-1}\left(\mathrm{M_2(1)}\right)\right|f(\mathrm{M_2(1)})+\left|\Psi_{G_{n}}^{-1}\left(\mathrm{M_2(2)}\right)\right|f(\mathrm{M_2(2)})+\left|\Psi_{G_{n}}^{-1}\left(\mathrm{M_2(3)}\right)\right|f(\mathrm{M_2(3)})\\ 	
				=&\prod_{i=1}^{4}(n_1-i)f(\mathrm{U})+\frac{1}{2}\prod_{i=1}^{5}(n_1-i)f(\mathrm{V})+\prod_{i=1}^{4}(n_1-i)f(\mathrm{W(1)})+\frac{1}{2}\prod_{i=1}^{4}(n_1-i)f(\mathrm{W(2)})\\&+\left[f(\mathrm{X'_{1}})-f(\mathrm{X'_{2}})+f(\mathrm{Y'_{1}})-f(\mathrm{Y'_{2}})\right]L_{4}+\left[2f(\mathrm{Y'_{1}})-f(\mathrm{X'_{2}})-f(\mathrm{Y'_{2}})\right]L_{3}\\&+\left[3f(\mathrm{Y'_{1}})-f(\mathrm{X'_{2}})-f(\mathrm{Y'_{2}})\right]L_{2}+(m-L_{1})f(\mathrm{X'_{2}})+(n_2-1-5m+4L_{1})f(\mathrm{Y'_{1}})\\&+(m-L_{1})f(\mathrm{Y'_{2}})+L_{1}\left(m-\sum_{i=1}^{3}L_{i}\right)f(\mathrm{Z'_{1}(0)})+L_{2}\left(m-\sum_{i=1}^{2}L_{i}\right)f(\mathrm{Z'_{1}(1)})\\&+L_{3}\left(m-1-L_{1}\right)f(\mathrm{Z'_{1}(2)})+L_{1}L_{3}f(\mathrm{Z'_{2}(0)})+\frac{1}{2}L_{2}(L_{2}-1)f(\mathrm{Z'_{2}(1)})\\&+\left(m-\sum_{i=1}^{3}L_{i}\right)(m-L_{1}-1)f(\mathrm{Z'_{3}(1)})+\frac{1}{2}(m-L_{1}-L_{2})(m-1-L_{1}-L_{2})f(\mathrm{Z'_{3}(2)})\\&+\left(n_{1}-1\right)\left(m-L_{1}-L_{2}-L_{3}\right)f(\mathrm{M_1(1)})+\left(n_{1}-1\right)\left(n_{1}-2\right)\left(m-L_{1}-L_{2}\right)f(\mathrm{M_1(2)})\\&+\prod_{i=1}^{3}(n_1-i)\left(m-L_{1}\right)f(\mathrm{M_1(3)})+(n_1-1)L_{3}f(\mathrm{M_2(1)})\\&+\left(n_{1}-1\right)\left(n_{1}-2\right)L_{2}f(\mathrm{M_2(2)})+\prod_{i=1}^{3}(n_1-i)L_{1}f(\mathrm{M_2(3)})\\
				=&\prod_{i=1}^{4}(n_1-i)f(\mathrm{U})+\frac{1}{2}\prod_{i=1}^{5}(n_1-i)f(\mathrm{V})+\prod_{i=1}^{4}(n_1-i)f(\mathrm{W(1)})+\frac{1}{2}\prod_{i=1}^{4}(n_1-i)f(\mathrm{W(2)})\\
&+\left[f(\mathrm{X'_{1}})-f(\mathrm{X'_{2}})+f(\mathrm{Y'_{1}})-f(\mathrm{Y'_{2}})\right]L_{4}+\left[2f(\mathrm{Y'_{1}})-f(\mathrm{X'_{2}})-f(\mathrm{Y'_{2}})\right.\\
&\quad\left.+L_1\left(f(\mathrm{M_3(0)})-f(\mathrm{M_1(0)})\right)+(m-1-L_1)\left(f(\mathrm{M_1(2)})-f(\mathrm{M_3(1)})\right)\right.\\
&\quad\left.+(n_1-1)\left(f(\mathrm{M_2(1)})-f(\mathrm{M_1(1)})\right)\right]L_{3}
+\left[3f(\mathrm{Y'_{1}})-f(\mathrm{X'_{2}})-f(\mathrm{Y'_{2}})\right.\\
&\quad\left.+\left(n_{1}-1\right)\left(n_{1}-2\right)\left(f(\mathrm{M_2(2)})-f(\mathrm{M_1(2)})\right)-L_1\left(f(\mathrm{Z_1(0)})+f(\mathrm{Z_1(1)})\right)\right.\\
				&\quad\left.-(m-L_1-1)f(\mathrm{Z_3(1)})-\left(n_{1}-1\right)f(\mathrm{M_1(1)})\right]L_{2}+L_{2}(m-L_1-L_2)f(\mathrm{Z_{1}(1)})\\
&+\frac{1}{2}L_2(L_2-1)f(\mathrm{Z_{2}(1)})+\frac{1}{2}(m-L_1-L_2)(m-L_1-L_2-1)f(\mathrm{Z_{3}(2)})\\
&+\left[f(\mathrm{X'_{2}})+4f(\mathrm{Y'_{1}})-f(\mathrm{Y'_{2}})+mf(\mathrm{Z_{1}(0)})-(n_1-1)f(\mathrm{M_{1}(1)})-(n_1-1)(n_1-2)f(\mathrm{M_{2}(2)})\right.\\
&\quad\left.+(n_1-1)(n_1-2)(n_1-3)\left(f(\mathrm{M_2(3)})-f(\mathrm{M_1(3)})\right)\right]L_{1}-L^2_1f(\mathrm{Z_1(0)})\\
&-(m-L_1)(m-L_1-1)f(\mathrm{M_3(1)})+mf(\mathrm{X'_{2}})+(n_2-1-5m)f(\mathrm{Y'_{1}})+mf(\mathrm{Y'_{2}})\\
&+m(n_1-1)f(\mathrm{M_1(1)})+m(n_1-1)(n_1-2)f(\mathrm{M_1(2)})+(n_1-1)(n_1-2)(n_1-3)f(\mathrm{M_1(3)})\\
				=&\overline{\lambda^{\prime}}\left[n_1,n_2,m,L_{1},L_2,f(\mathrm{X'_{1}}),f(\mathrm{X'_{2}}),f(\mathrm{Y'_{1}}),f(\mathrm{Y'_{2}}),f(\mathrm{Z'_{1}(0)}),f(\mathrm{Z'_{1}(1)}),f(\mathrm{Z'_{2}(0)}),f(\mathrm{Z'_{2}(1)}),\right.\\
				&\quad\left.f(\mathrm{Z'_{3}(1)}),f(\mathrm{Z'_{3}(2)}),f(\mathrm{M_1(1)}), f(\mathrm{M_1(2)}),f(\mathrm{M_1(3)}),f(\mathrm{M_2(1)}),f(\mathrm{M_2(2)}),f(\mathrm{M_2(3)}),f(\mathrm{U}),\right.\\
				&\quad\left.f(\mathrm{V}),f(\mathrm{W(1)}),f(\mathrm{W(2)})\right]+\left[f(\mathrm{X'_{1}})-f(\mathrm{X'_{2}})+f(\mathrm{Y'_{1}})-f(\mathrm{Y'_{2}})\right]L_{4}\\
			\end{aligned}
		\end{equation*}
		
		If $h\geqslant5$ is odd, then
		\begin{align*}
			^{h}I_f(G_{n})=&\sum_{v_{0}v_{1}\cdots v_{h}\in\mathcal{P}_{h}}\omega(d(v_0),d(v_1),\cdots,d(v_h))f\left(d(v_0),d(v_1),\cdots,d(v_h)\right)\\
			=&\left|\Psi_{G_{n}}^{-1}\left(\mathrm{U}\right)\right|f(\mathrm{U})+\left|\Psi_{G_{n}}^{-1}\left(\mathrm{V}\right)\right|f(\mathrm{V})+\sum_{a=1}^{\frac{h-1}{2}}\left|\Psi_{G_{n}}^{-1}\left(\mathrm{W(a)}\right)\right|f\left(\mathrm{W(a)}\right)+^{h}I_f(S_{n_2})\\
		&+\sum_{a=1}^{h-1}\left|\Psi_{G_{n}}^{-1}\left(\mathrm{M_1(a)}\right)\right|f\left(\mathrm{M_1(a)}\right)
			+\sum_{a=1}^{h-1}\left|\Psi_{G_{n}}^{-1}\left(\mathrm{M_2(a)}\right)\right|f\left(\mathrm{M_2(a)}\right)\\
			=&\prod_{i=1}^{h}(n_1-i)f(\mathrm{U})+\frac{1}{2}\prod_{i=1}^{h+1}(n_1-i)f(\mathrm{V})
			+\sum_{a=1}^{\frac{h-1}{2}}\left[\prod\limits_{i=1}^{a}(n_1-i)\prod\limits_{i=1}^{h-a}(n_1-a-i)\right]f\left(\mathrm{W(a)}\right)\\
			&+^{h}I_f(S_{n_2})+\sum_{a=1}^{h-1}\left[\prod_{i=1}^{a}\left(n_{1}-i\right)L_{h-a}\right]f\left(\mathrm{M_2(a)}\right)\\
			&+\sum_{a=1}^{h-1}\left[\prod_{i=1}^{a}\left(n_{1}-i\right)\left(m-\sum_{i=1}^{h-a}L_{i}\right)\right]f\left(\mathrm{M_1(a)}\right)\\
			=&\lambda\left(n_1,n_2,m,h,L_{1},\cdots,L_{h-3}\right)+\left[f(\mathrm{X'_{1}})-f(\mathrm{X'_{2}})+f(\mathrm{Y'_{1}})-f(\mathrm{Y'_{2}})\right]L_{h}\\&+\left[2f(\mathrm{Y'_{1}})-f(\mathrm{X'_{2}})-f(\mathrm{Y'_{2}})+(n_1-1)\left(f(\mathrm{M_2(1)})-f(\mathrm{M_1(1)})\right)\right]L_{h-1}\\&+\left[3f(\mathrm{Y'_{1}})-f(\mathrm{X'_{2}})-f(\mathrm{Y'_{2}})+\left(n_{1}-1\right)\left(n_{1}-2\right)\left(f(\mathrm{M_2(2)})-f(\mathrm{M_1(2)})\right)\right.\\
			&\quad\left.-\left(n_{1}-1\right)f(\mathrm{M_1(1)})\right]L_{h-2}\\
			=&\lambda^{\prime}\left[n_1,n_2,m,h,L_{1},\cdots,L_{h-3},f(\mathrm{X'_{1}}),f(\mathrm{X'_{2}}),f(\mathrm{Y'_{1}}),f(\mathrm{Y'_{2}}),f(\mathrm{Z'_{1}(a)}),f(\mathrm{Z'_{2}(a)}),f(\mathrm{Z'_{3}(a)}),\right.\\
			&\quad\left.f(\mathrm{M_1(a)}), f(\mathrm{M_2(a)}),f(\mathrm{U}),f(\mathrm{V}),f(\mathrm{W(a)})\right] +\left[f(\mathrm{X'_{1}})-f(\mathrm{X'_{2}})+f(\mathrm{Y'_{1}})-f(\mathrm{Y'_{2}})\right]L_{h}\\&+\left[2f(\mathrm{Y'_{1}})-f(\mathrm{X'_{2}})-f(\mathrm{Y'_{1}})+L_1\left(f(\mathrm{Z'_2(0)})-f(\mathrm{Z'_1(0)})+f(\mathrm{Z'_3(1)})-f(\mathrm{Z'_1(h-2)})\right)\right.\\
			&\quad\left.+(m-1)\left(f(\mathrm{Z'_1(h-2)})-f(\mathrm{Z'_3(1)})\right)+(n_1-1)\left(f(\mathrm{M_2(1)})-f(\mathrm{M_1(1)})\right)\right]L_{h-1}\\&+\left[3f(\mathrm{Y'_{1}})-f(\mathrm{X'_{2}})-f(\mathrm{Y'_{2}})+L_1\left(f(\mathrm{Z'_3(1)})-f(\mathrm{Z'_1(0)})+f(\mathrm{Z'_3(2)})-f(\mathrm{Z'_1(h-3)})\right)\right.\\
			&\quad\left.+L_2\left(f(\mathrm{Z'_2(1)})-f(\mathrm{Z'_1(1)})+f(\mathrm{Z'_3(2)})-f(\mathrm{Z'_1(h-3)})\right)\right.\\
			&\quad\left.+(m-1)\left(f(\mathrm{Z'_1(h-3)})-f(\mathrm{Z'_3(1)})-f(\mathrm{Z'_3(2)})\right)\right.\\
			&\quad\left.+(n_1-1)(n_1-2)\left(f(\mathrm{M_2(2)})-f(\mathrm{M_1(2)})\right)-(n_1-1)f(\mathrm{M_1(1)})\right]L_{h-2}\\
			=&\overline{\lambda^{\prime}}\left[n_1,n_2,m,h,L_{1},\cdots,L_{h-1},f(\mathrm{X'_{1}}),f(\mathrm{X'_{2}}),f(\mathrm{Y'_{1}}),f(\mathrm{Y'_{2}}),f(\mathrm{Z'_{1}(a)}),f(\mathrm{Z'_{2}(a)}),\right.\\
			&\quad\left.f(\mathrm{Z'_{3}(a)}),f(\mathrm{M_1(a)}), f(\mathrm{M_2(a)}),f(\mathrm{U}),f(\mathrm{V}),f(\mathrm{W(a)})\right]+\overline{\mu^{\prime}}L_{h}
		\end{align*}
		
		If $h\geqslant6$ is even, then
		\begin{equation*}		
			\begin{aligned}
				^{h}I_f(G_{n})=&\sum_{v_{0}v_{1}\cdots v_{h}\in\mathcal{P}_{h}}\omega(d(v_0),d(v_1),\cdots,d(v_h))f\left(d(v_0),d(v_1),\cdots,d(v_h)\right)\\
				=&\left|\Psi_{G_{n}}^{-1}\left(\mathrm{U}\right)\right|f(\mathrm{U})+\left|\Psi_{G_{n}}^{-1}\left(\mathrm{V}\right)\right|f(\mathrm{V})
				+\sum_{a=1}^{\frac{h-1}{2}}\left|\Psi_{G_{n}}^{-1}\left(\mathrm{W(a)}\right)\right|f\left(\mathrm{W(a)}\right)+^{h}I_f(S_{n_2})\\
				&+\sum_{a=1}^{h-1}\left|\Psi_{G_{n}}^{-1}\left(\mathrm{M_1(a)}\right)\right|f\left(\mathrm{M_1(a)}\right)
				+\sum_{a=1}^{h-1}\left|\Psi_{G_{n}}^{-1}\left(\mathrm{M_2(a)}\right)\right|f\left(\mathrm{M_2(a)}\right)\\
				=&\prod_{i=1}^{h}(n_1-i)f(\mathrm{U})+\frac{1}{2}\prod_{i=1}^{h+1}(n_1-i)f(\mathrm{V})
				+\sum_{a=1}^{\frac{h-3}{2}}\left[\prod\limits_{i=1}^{a}(n_1-i)\prod\limits_{i=1}^{h-a}(n_1-a-i)\right]f\left(\mathrm{W(a)}\right)\\
				&+\dfrac{h}{2}\prod\limits_{i=1}^{\frac{h}{2}}(n_1-i)\prod\limits_{i=1}^{\frac{h}{2}}(n_1-\frac{h}{2}-i)f(\mathrm{W(\frac{h}{2})})
				+\lambda\left(n_1,n_2,m,L_{1},\cdots,L_{h-1}\right) \\
				&+\left[f(\mathrm{X'_1})-f(\mathrm{X'_2})+f(\mathrm{Y'_1})-f(\mathrm{Y'_2})\right]L_{h}+\left[2f(\mathrm{Y'_1})-f(\mathrm{X'_2})-f(\mathrm{Y'_2})\right]L_{h-1}\\
				&+\left[3f(\mathrm{Y'_1})-f(\mathrm{X'_2})-f(\mathrm{Y'_2})+1-m\right]L_{h-2}+\sum_{a=1}^{h-1}\left[\prod_{i=1}^{h}\left(n_{1}-i\right)L_{h-k}\right]f\left(\mathrm{M_1(a)}\right)\\
				&+\sum_{a=1}^{h-1}\left[\prod_{i=1}^{h}\left(n_{1}-i\right)\left(m-\sum_{i=1}^{h-k}L_{i}\right)\right]f\left(\mathrm{M_2(a)}\right)\\
				=&\lambda^{\prime}\left[n_1,n_2,m,h,L_{1},\cdots,L_{h-3},f(\mathrm{X'_{1}}),f(\mathrm{X'_{2}}),f(\mathrm{Y'_{1}}),f(\mathrm{Y'_{2}}),f(\mathrm{Z'_{1}(a)}),f(\mathrm{Z'_{2}(a)}),f(\mathrm{Z'_{3}(a)}),\right.\\
				&\quad\left.f(\mathrm{M_1(a)}), f(\mathrm{M_2(a)}),f(\mathrm{U}),f(\mathrm{V}),f(\mathrm{W(a)})\right] +\left[f(\mathrm{X'_{1}})-f(\mathrm{X'_{2}})+f(\mathrm{Y'_{1}})-f(\mathrm{Y'_{2}})\right]L_{h}\\&+\left[2f(\mathrm{Y'_{1}})-f(\mathrm{X'_{2}})-f(\mathrm{Y'_{1}})+L_1\left(f(\mathrm{Z'_2(0)})-f(\mathrm{Z'_1(0)})+f(\mathrm{Z'_3(1)})-f(\mathrm{Z'_1(h-2)})\right)\right.\\
				&\quad\left.+(m-1)\left(f(\mathrm{Z'_1(h-2)})-f(\mathrm{Z'_3(1)})\right)+(n_1-1)\left(f(\mathrm{M_2(1)})-f(\mathrm{M_1(1)})\right)\right]L_{h-1}\\&+\left[3f(\mathrm{Y'_{1}})-f(\mathrm{X'_{2}})-f(\mathrm{Y'_{2}})+L_1\left(f(\mathrm{Z'_3(1)})-f(\mathrm{Z'_1(0)})+f(\mathrm{Z'_3(2)})-f(\mathrm{Z'_1(h-3)})\right)\right.\\
				&\quad\left.+L_2\left(f(\mathrm{Z'_2(1)})-f(\mathrm{Z'_1(1)})+f(\mathrm{Z'_3(2)})-f(\mathrm{Z'_1(h-3)})\right)\right.\\
				&\quad\left.+(m-1)\left(f(\mathrm{Z'_1(h-3)})-f(\mathrm{Z'_3(1)})-f(\mathrm{Z'_3(2)})\right)\right.\\
				&\quad\left.+(n_1-1)(n_1-2)\left(f(\mathrm{M_2(2)})-f(\mathrm{M_1(2)})\right)-(n_1-1)f(\mathrm{M_1(1)})\right]L_{h-2}\\
				=&\overline{\lambda^{\prime}}\left[n_1,n_2,m,h,L_{1},\cdots,L_{h-1},f(\mathrm{X'_{1}}),f(\mathrm{X'_{2}}),f(\mathrm{Y'_{1}}),f(\mathrm{Y'_{2}}),f(\mathrm{Z'_{1}(a)}),f(\mathrm{Z'_{2}(a)}),\right.\\
				&\quad\left.f(\mathrm{Z'_{3}(a)}),f(\mathrm{M_1(a)}), f(\mathrm{M_2(a)}),f(\mathrm{U}),f(\mathrm{V}),f(\mathrm{W(a)})\right]+\overline{\mu^{\prime}}L_{h}\\
			\text{where } \overline{\mu^{\prime}}=&f(\mathrm{X'_{1}})-f(\mathrm{X'_{2}})+f(\mathrm{Y'_{1}})-f(\mathrm{Y'_{2}}).
			\end{aligned}
		\end{equation*}
	\end{proof}

	\section{The conditions on $f$ determining a graph by $^{h}I_f$}
	
	In this section, we will give conditions on $f: \bigcup_{i=1}^{\infty} \mathbb{Z}_{+}^{i} \rightarrow \mathbb{R}$ for which the corresponding higher-order invariant $^{h}I_f$ can determine a graph in some graph families.
	
	First, we consider the starlike trees. Let $\mathcal{G}_{n}$ be the set of all starlike trees with $n$ vertices. $S_n, S'_n\in\mathcal{G}_{n}$ are two starlike trees whose the maximum degree is $m$ and $m'$, respectively. For which function $f: \bigcup_{i=1}^{\infty} \mathbb{Z}_{+}^{i} \rightarrow \mathbb{R}$, can the corresponding higher-order invariant $^{h}I_f$ distinguish the graphs in $\mathcal{G}_{n}$, i.e., $S_n\cong S'_n$ if and only if $^{h}I_f(S_{n})=^{h}I_f(S^{\prime}_{n})$ for all $0\leqslant h\leqslant \rho$, where $S_n, S'_n\in\mathcal{G}_{n}$ ?
	
	The following result provides a sufficient condition on $f$ determining a graph in $\mathcal{G}_{n}$ by $^{h}I_f$.
	
	\begin{theorem}\label{t-7}
		If $f:\bigcup_{i=1}^{\infty} \mathbb{Z_{+}}^{i} \rightarrow \mathbb{R}$ with $f(x_1,x_2,\cdots,x_i)=f(x_i,x_{i-1},\cdots,x_1)$ for $i=1,2,\cdots$, and
		
		$(a)$ $\frac{f(x)-f(y)}{x-y}\neq f(2)-f(1)$ for $x\neq y$ and $x,y\geqslant3$;
		
		$(b)$ $g_{t}(x)\neq g_{t}(2)$ for $x\neq 2$ and $t=0,1,\cdots$, where $g_{t}(x)=f(x,\underbrace{2,\cdots,2}_{t},1)-f(x,\underbrace{2,\cdots,2}_{t+1})$.\\
		Then the corresponding higher-order invariant $^{h}I_f$ can distinguish the graphs in $\mathcal{G}_{n}$, i.e., $S_n\cong S'_n$ if and only if $^{h}I_f(S_{n})=^{h}I_f(S^{\prime}_{n})$ for all $h\geq 0$, where $S_n, S'_n\in\mathcal{G}_{n}$.
	\end{theorem}
	
	\begin{proof}
		If $S_{n}\cong S^{\prime}_{n}$, then it is clearly that $^{h}I_f(S_{n})=^{h}I_f(S^{\prime}_{n})$.
		
		On the other hand, let $^{h}I_f(S_{n})=^{h}I_f(S^{\prime}_{n})$ for all $h\geqslant 0$, then we know from the proof of Theorem \ref{t-2},
		$$^{0}I_f(S_{n})=\sum_{v\in V(S_{n})}\omega(d(v))f(d(v))=f(m)+mf(1)+(n-m-1)f(2).$$
		And
		$$f(m)+m(f(1)-f(2))+(n-1)f(2)=f(m')+m'(f(1)-f(2))+(n-1)f(2)$$
		i.e.,
		$$f(m)-f(m')=(m-m')(f(2)-f(1))$$
		by $^{0}I_f(S_{n})=^{0}I_f(S^{\prime}_{n})$. We can get $m=m'$ since $\frac{f(x)-f(y)}{x-y}\neq f(2)-f(1)$ for $x\neq y$.
		
		Again,
		$$^{1}I_f(S_{n})=\overline{\lambda}(m,n)+\left[f(m,1)-f(m,2)+f(2,2)-f(1,2)\right]L_{1}(S_{n})$$
		from the proof of Theorem \ref{t-2}. Using $^{1}I_f(S_{n})=^{1}I_f(S^{\prime}_{n})$, we have
		\begin{equation*}
			\begin{aligned}
				&\overline{\lambda}(m,n)+\left[f(m,1)-f(m,2)+f(2,2)-f(1,2)\right]L_{1}(S_{n})\\
				=&\overline{\lambda}(m',n)+\left[f(m',1)-f(m',2)+f(2,2)-f(1,2)\right]L_{1}(S'_{n}).
			\end{aligned}
		\end{equation*}
		And $L_{1}(S_{n})=L_{1}(S'_{n})$ since $m=m'>2$ and $g_{0}(x)=f(x,1)-f(x,2)\neq g_{0}(2)$ for $x\neq 2$.
		
		From Theorem \ref{t-2}, $^{h}I_f(S_{n})=\overline{\lambda}+\overline{\mu}L_{h}(S_{n})$, where
		$\overline{\lambda}=\overline{\lambda}[n,m,L_{1},\cdots,L_{h-1},f(X_{1}),f(X_{2})$, $f(Y_{1}),f(Y_{2}),f(Z_1(a)),f(Z_2(a)),f(Z_3(a))]$ and
		\begin{align*}
			\overline{\mu}&=f(X_{1})-f(X_{2})+f(Y_{1})-f(Y_{2})\\
			&=f(m,\underbrace{2,\cdots,2}_{h-1},1)-f(m,\underbrace{2,\cdots,2}_{h})+f(\underbrace{2,\cdots,2}_{h+1})-f(1,\underbrace{2,\cdots,2}_{h})\\
			&=g_{h-1}(m)-g_{h-1}(2)\neq 0
		\end{align*}
		Let $L_{1}(S_{n})=L_{1}(S'_{n}),L_{2}(S_{n})=L_{2}(S'_{n}),L_{h-1}(S_{n})=L_{h-1}(S'_{n})$. By $^{h}I_f(S_{n})=^{h}I_f(S^{\prime}_{n})$, we can get $L_{h}(S_{n})=L_{h}(S'_{n})$. 	
		
		Hence $S_{n}$ and $S_{n}^{\prime}$ are isomorphic.
	\end{proof}
	
	Next, we consider the generalized starlike trees. Let $\mathcal{G}_{n,r}$ be the set of all generalized starlike trees with $n$ vertices and  maximum degree $r$.
	
	\begin{theorem}\label{t-8}
		If $f:\bigcup_{i=1}^{\infty} \mathbb{Z_{+}}^{i} \rightarrow \mathbb{R}$ with $f(x_1,x_2,\cdots,x_i)=f(x_i,x_{i-1},\cdots,x_1)$ for $i=1,2,\cdots$, and
		
		$(a)$ $\frac{xf(x)-yf(y)}{x-y}\neq f(1)$ for $x\neq y$ and $x,y \geqslant3$;
		
		$(b)$ $g_{t}(x)\neq g_{t}(2)$ for $x\neq 2$ and $t=0,1,\cdots$, where $g_{t}(x)=f(x,\underbrace{2,\cdots,2}_{t},1)-f(x,\underbrace{2,\cdots,2}_{t+1})$.\\
		Then the corresponding higher-order invariant $^{h}I_f$ can distinguish the graphs in $\mathcal{G}'_{n,r}$, i.e., $G_n\cong G'_n$ if and only if $^{h}I_f(G_{n})=^{h}I_f(G^{\prime}_{n})$ for all $h\geqslant 0$, where $G_n, G'_n\in\mathcal{G}'_{n,r}$.
	\end{theorem}
	
	\begin{proof}
		If $G_{n}\cong G^{\prime}_{n}$, then it is clearly that $^{h}I_f(G_{n})=^{h}I_f(G^{\prime}_{n})$.
		
		Let $G_n=K_{n_{1}}(u_{0}=v_{0})S_{n_{2}}, G'_n=K_{n'_{1}}(u'_{0}=v'_{0})S_{n'_{2}}\in \mathcal{G}_{n,r}$, where $n=n_1+n_2-1=n'_1+n'_2-1$ and $r=n_1+m-1=n'_1+m'-1$, $m, m'$ are the degrees of roots of $S_{n_2}$ and $S_{n'_2}$, respectively.
		
		If $^{h}I_f(G_{n})=^{h}I_f(G^{\prime}_{n})$ for all $h\geqslant 0$, then
		\begin{align*}
			^{0}I_f(G_{n})&=\sum_{v\in V(G_{n})}\omega(d(v))f(d(v))\\&=mf(1)+f(m+n_1-1)+(n-m-n_1)f(2)+(n_1-1)f(n_1-1)\\
			&=(n_1-1)f(n_1-1)-(n_1-1)f(1)+rf(1)+f(r)+(n-r-1)f(2)\\
		\end{align*}
		from the proof of Theorem \ref{t-6}. By $^{0}I(G_{n})=^{0}I(G^{\prime}_{n})$, we have
		\begin{align*} (n_1-1)f(n_1-1)-(n_1-1)f(1)+rf(1)+f(r)+(n-r-1)f(2)\\=(n'_1-1)f(n'_1-1)-(n'_1-1)f(1)+rf(1)+f(r)+(n-r-1)f(2)
		\end{align*}
		i.e., $$f(1)[(n'_1-1)-(n_1-1)]=(n'_1-1)f(n'_1-1)-(n_1-1)f(n_1-1).$$
		We can obtain $n_1=n'_1$ since $\frac{xf(x)-yf(y)}{x-y}\neq f(1)$ for $x\neq y$, and $n_2=n'_2$ by $n=n_1+n_2-1=n'_1+n'_2-1$, $m=m'$ by $r=n_1+m-1=n'_1+m'-1$.
		
		Again,
		\begin{equation*}
			\begin{aligned}
				^{1}I_f(G_{n})=&\overline{\lambda^{\prime}}\left[m,n_{1},n_{2},f(m+n_{1}-1,1),f(m+n_{1}-1,2),f(1,2),f(2,2)\right]\\
				&+\left[f(m+n_{1}-1,1)-f(m+n_{1}-1,2)-f(1,2)+f(2,2)\right]L_{1}(G_n)
			\end{aligned}
		\end{equation*}
		from the proof of Theorem \ref{t-6}. Using $^{1}I_f(G_{n})=^{1}I_f(G^{\prime}_{n})$, we have
		\begin{equation*}
			\begin{aligned}
				&\overline{\lambda^{\prime}}\left[m,n_{1},n_{2},f(m+n_{1}-1,1),f(m+n_{1}-1,2),f(1,2),f(2,2)\right]\\
				&+\left[f(m+n_{1}-1,1)-f(m+n_{1}-1,2)-f(1,2)+f(2,2)\right]L_{1}(G_n)\\
				=&\overline{\lambda^{\prime}}\left[m',n'_{1},n'_{2},f(m'+n'_{1}-1,1),f(m'+n'_{1}-1,2),f(1,2),f(2,2)\right]\\
				&+\left[f(m'+n'_{1}-1,1)-f(m'+n'_{1}-1,2)-f(1,2)+f(2,2)\right]L_{1}(G'_n).
			\end{aligned}
		\end{equation*}
		And $L_{1}(G_{n})=L_{1}(G'_{n})$ since $m=m'>2$ and $g_{0}(x)=f(x,1)-f(x,2)\neq g_{0}(2)$ for $x\neq 2$.
		\begin{align*}
			&\text{From Theorem \ref{t-6},} ^{h}I_f(G_{n})=\overline{\lambda^{'}}+\overline{\mu^{'}}L_{h}(G_{n}), \text{where }\\
			&\overline{\lambda'}=\overline{\lambda'}\left[n_1,n_2,m,h,L_{1},\cdots,L_{h-1},f(\mathrm{X'_{1}}),f(\mathrm{X'_{2}}),f(\mathrm{Y'_{1}}),f(\mathrm{Y'_{2}}),f(\mathrm{Z'_{1}(a)}),f(\mathrm{Z'_{2}(a)}),f(\mathrm{Z'_{3}(a)}),\right.\\
			&\qquad\quad \left.f(\mathrm{M_1(a)}), f(\mathrm{M_2(a)}),f(\mathrm{U}),f(\mathrm{V}),f(\mathrm{W(a)})\right]
		\end{align*}
		and
		\begin{align*}
			\overline{\mu'}&=f(X'_{1})-f(X'_{2})+f(Y'_{1})-f(Y'_{2})\\
			&=f(m,\underbrace{2,\cdots,2}_{h-1},1)-f(m,\underbrace{2,\cdots,2}_{h})+f(\underbrace{2,\cdots,2}_{h+1})-f(1,\underbrace{2,\cdots,2}_{h})\\
			&=g_{h-1}(m)-g_{h-1}(2)\neq 0
		\end{align*}
		Let $L_{1}(G_{n})=L_{1}(G'_{n}), L_{2}(G_{n})=L_{2}(G'_{n}), L_{h-1}(G_{n})=L_{h-1}(G'_{n})$. By $^{h}I_f(G_{n})=^{h}I_f(G^{\prime}_{n})$, we can get $L_{h}(G_{n})=L_{h}(G'_{n})$. 	
		
		Hence, $G_{n}\cong G_{n}^{\prime}$.
		
	\end{proof}

	\section{Applications}
	
	In the previous section we give some conditions for the function $f$ that allow the corresponding higher-order invariants $^{h}I_f$ to determine graphs in some graph families. In this section, we will give some concrete examples, i.e., there are some higher-order topological indices satisfy the sufficient conditions.
	
	\subsection{The higher-order connectivity index}
	
	The first degree-based topological index was put forward in 1975 by Randi\'{c} in his paper \cite{delorme2002randic}. This index was defined as
	$$\xi(G)=\sum_{xy\in E(G)}\frac{1}{\sqrt{d(x)d(y)}}.$$
	
	Let $f:\bigcup_{i=1}^{\infty} \mathbb{Z_{+}}^{i} \rightarrow \mathbb{R}$ with $f(x_1,x_2,\cdots,x_i)=\frac{1}{\sqrt{x_{1}x_{2}\cdots x_{i}}}$ for $i=1,2,\cdots$, then the corresponding higher-order invariants $^{h}I_f$ is just the $h$-connectivity index, see \cite{rada2002higher}, defined as
	$$^{h}\chi(G)=\sum\limits_{v_{0}v_{1}v_{2}\cdots v_{h}}\frac{1}{\sqrt{d_{0}d_{1}\cdots d_{h}}},$$
	where the sum runs over all paths $v_{0}v_{1}v_{2}\cdots v_{h}$ of length $h$ and $d_{i}=d(v_i)$ is the degree of vertex $v_i$ in $G$.
	
	Let's check that function $f$ satisfies the condition of Theorem \ref{t-7}:
	
	(a) $f(x)-f(y)=f'(\xi)(x-y)$, where $\xi>2$ since $x,y\geqslant 2$, then $\frac{f(x)-f(y)}{x-y}=f'(\xi)=-\frac{1}{2}\frac{1}{\xi\sqrt{\xi}}>-\frac{1}{4\sqrt{2}}>f(2)-f(1)=\frac{1}{\sqrt{2}}-1$. And $\frac{f(x)-f(y)}{x-y}\neq f(2)-f(1)$ for $x\neq y$;
	
	(b) $g_{t}(x)=\frac{1}{\sqrt{2^tx}}-\frac{1}{\sqrt{2^{t+1}x}}\neq \frac{1}{\sqrt{2^{t+1}}}-\frac{1}{\sqrt{2^{t+2}}}=g_{t}(2)$ for $x\neq 2$ and $t=0,1,\cdots$.
	
	By Theorem \ref{t-7}, the following result is immediate.
	
	\begin{corollary}\cite{cai2024}
		Let $S_n, S'_n$  be starlike trees with $n$ vertices. Then $S_n\cong S'_n$ if and only if $^{h}\chi(S_{n})=^{h}\chi(S^{\prime}_{n})$ for all $h\geq 0$.
	\end{corollary}
	
	Similarly, the function $f$ also satisfies the condition of Theorem \ref{t-8}, since $\frac{xf(x)-yf(y)}{x-y}=\frac{1}{2\sqrt{\xi}}\neq 1=f(1)$, where $\xi>2$. By Theorem \ref{t-8}, the following result is immediate.
	
	\begin{corollary}
		Let $G_n, G'_n\in\mathcal{G}'_{n,r}$ be generalized starlike trees with $n$ vertices and maximal degree $r$. Then $G_n\cong G'_n$  if and only if $^{h}\chi(G_{n})=^{h}\chi(G^{\prime}_{n})$ for all $h\geqslant 0$.
	\end{corollary}

	\subsection{The higher-order sum-connectivity index}
	
	The sum-connectivity index (see \cite{bzhou2009})
	$$\chi_s(G)=\sum\limits_{xy\in E(G)}\frac{1}{\sqrt{d_{x}+d_{y}}}.$$
	Imitating the higher-order connectivity index, we introduce the higher-order sum-connectivity index as
	$$^{h}\chi_s(G)=\sum\limits_{v_{0}v_{1}v_{2}\cdots v_{h}}\frac{1}{\sqrt{d_{0}+d_{1}+\cdots +d_{h}}}$$
	where the sum runs over all paths $v_{0}v_{1}v_{2}\cdots v_{h}$ of length $h$ and $d_{i}=d(v_i)$ is the degree of vertex $v_i$ in $G$, i.e., $^{h}\chi_s=^{h}I_f$ with $f:\bigcup_{i=1}^{\infty} \mathbb{Z_{+}}^{i} \rightarrow \mathbb{R}$ and $f(x_1,x_2,\cdots,x_i)=\frac{1}{\sqrt{x_{1}+x_{2}+\cdots +x_{i}}}$ for $i=1,2,\cdots$.
	
	It can easily be checked that this function $f$ satisfies the condition of Theorem \ref{t-7} and Theorem \ref{t-8}, and the following results are directly derived from them.
	
	\begin{corollary}
		Let $S_n, S'_n$  be starlike trees with $n$ vertices. Then $S_n\cong S'_n$ if and only if $^{h}\chi_s(S_{n})=^{h}\chi_s(S^{\prime}_{n})$ for all $h\geqslant 0$.
	\end{corollary}
	
	\begin{corollary}
		Let $G_n, G'_n\in\mathcal{G}'_{n,r}$ be generalized starlike trees with $n$ vertices and maximal degree $r$. Then $G_n\cong G'_n$  if and only if $^{h}\chi_s(G_{n})=^{h}\chi_s(G^{\prime}_{n})$ for all $h\geqslant 0$.
	\end{corollary}

	\subsection{The higher-order Second hyper-Zagreb index}    	
	
	The Zagreb indices were first introduced by Gutman \cite{gutman1972graph} and found to closely correlate the total $\pi$-electron energy of alternant hydrocarbons. Successful applicability of Zagreb indices in reticular chemistry motivated the researchers to introduce other variants of these Zagreb indices \cite{zhou2004zagreb,zeng2021open}. One of the variant known as the second hyper-Zagreb index is defined as
	$$M(G)=\sum_{xy\in E(G)}(d_{x}d_{y})^{2}.$$
	We generalize the second hyper-Zagreb index to the higher-order Second hyper-Zagreb index
	$$^{h}M(G)=\sum\limits_{v_{0}v_{1}v_{2}\cdots v_{h}}(d_{0}d_{1}\cdots d_{h})^2,$$
	where the sum runs over all paths $v_{0}v_{1}v_{2}\cdots v_{h}$ of length $h$ and $d_{i}=d(v_i)$ is the degree of vertex $v_i$ in $G$, i.e., $^{h}M=^{h}I_f$ with $f:\bigcup_{i=1}^{\infty} \mathbb{Z_{+}}^{i} \rightarrow \mathbb{R}$ and $f(x_1,x_2,\cdots,x_i)=(x_{1}x_{2}\cdots x_{i})^2$ for $i=1,2,\cdots$.
	
	It can easily be checked that this function $f$ satisfies the condition of Theorem \ref{t-7} and Theorem \ref{t-8}, and the following results are directly derived from them.
	
	\begin{corollary}
		Let $S_n, S'_n$  be starlike trees with $n$ vertices. Then $S_n\cong S'_n$ if and only if $^{h}M(S_{n})=^{h}M(S^{\prime}_{n})$ for all $h\geqslant 0$.
	\end{corollary}
	
	\begin{corollary}
		Let $G_n, G'_n\in\mathcal{G}'_{n,r}$ be generalized starlike trees with $n$ vertices and maximal degree $r$. Then $G_n\cong G'_n$  if and only if $^{h}M(G_{n})=^{h}M(G^{\prime}_{n})$ for all $h\geqslant 0$.
	\end{corollary}

	\section{Conclusion}
	
	We introduce the higher order invariant or higher order topological index of a graph based on the path sequence. The mathematical properties of the higher order invariant for some graphs are obtained. Some sufficient conditions are found, which can be used to determine every graph in some graph families by the higher order invariants. We hope to further solve the isomorphic problem of graphs with such higher order invariants.

	{\bf Acknowledgements} This work is supported by the National Natural Science Foundation of China (No.12201634) and the Hunan Provincial Natural Science Foundation of China (2020JJ4423,2023JJ30070). 

\end{document}